\documentclass[sn-mathphys,Numbered]{sn-jnl}% Math and Physical Sciences Reference Style

\usepackage{graphicx}%
\usepackage{multirow}%
\usepackage{amsmath,amssymb,amsfonts}%
\usepackage{amsthm}%
\usepackage{mathrsfs}%
\usepackage[title]{appendix}%
\usepackage{xcolor}%
\usepackage{textcomp}%
\usepackage{manyfoot}%
\usepackage{booktabs}%
\usepackage{algorithm}%
\usepackage{algorithmicx}%
\usepackage{algpseudocode}%
\usepackage{listings}%
\usepackage{float}
\usepackage{indentfirst}
\usepackage{bbm}
\usepackage{hyperref}
\usepackage{caption}
\usepackage{subfigure} 
\usepackage{epstopdf}
\usepackage{epsfig}

\allowdisplaybreaks[4]

\theoremstyle{thmstyleone}%
\newtheorem{theorem}{Theorem}%  meant for continuous numbers
\newtheorem{lemma}[theorem]{Lemma}%
\theoremstyle{thmstyletwo}%
\newtheorem{example}{Example}%
\newtheorem{remark}{Remark}%

\theoremstyle{thmstylethree}%

\begin{document}
\title[density function of solution to MVE]{Existence and smoothness of density function of solution to Mckean--Vlasov Equation with general coefficients}

\author[1]{\fnm{Boyu} \sur{Wang}}\email{wangby23@mails.jlu.edu.cn}

\author*[1]{\fnm{Yongkui} \sur{Zou}}\email{zouyk@jlu.edu.cn}

\author[1]{\fnm{Jinhui} \sur{Zhou}}\email{jhzhou21@mails.jlu.edu.cn}

%\author[1]{\fnm{Jinhui} \sur{Zhou}}\email{zhoujh21@mails.jlu.edu.cn}

\affil[1]{\orgdiv{School of Mathematics}, \orgname{Jilin University}, \orgaddress{\city{Changchun}, \postcode{130012}, \country{China}}}

\abstract{	
In this paper, we study the existence and smoothness of a density function to the solution of a Mckean-Vlasov equation with the aid of Malliavin calculus. We first show the existence of the density function under assumptions that the coefficients of equation are  only  Lipschitz continuity and satisfy a uniform elliptic condition.  Furthermore, we derive a precise regularity order and bounded a priori estimate for the density function under  optimal smoothness assumptions for the coefficients. Finally, we present several numerical experiments to illustrate the approximation of the density function independently determined by solving a Fokker-Planck equation.
}



\maketitle

\section{Introduction}
Mckean--Vlasov equation (MVE for short), where local sample trajectory properties depend on the global distribution of the solution process, is a particular type of stochastic differential equations (SDEs  for short). MVE characterizes the average limiting behavior of a large weakly interacting particle system as the particle number tends to infinity. They are often used to describe the dynamic of large-scale systems and have found extensive applications in a variety of fields, such as statistical physics \cite{Me1996}, mathematical biology \cite{BuCaMo2007}, deep learning neural networks \cite{HuReSiSz2021} and mean-field games \cite{CaDe2013,CaDe2018}, etc. 

An MVE was initially investigated in statistical physics by Kac \cite{Ka1956}, serving as a stochastic counterpart to the Vlasov kinetic equation of plasma.  The systematically probabilistic study of MVEs was started by McKean \cite{Mc1966}. Then, the fundamental properties  such as existence, uniqueness and stability of solutions were well established \cite{BaMeMe2020,MiVe2020,Ch2020,ChFr2022,Sz1991}.  Due to the lack of explicit solution to MVEs, the numerical approximation methods became an important tool for investigating MVEs and characterizing their developments \cite{LiMaSoWuYi2023,BoTa1997,KuNeReSt2022,BaHu2022}.

An important task of numerical approximation to MVE is  to explore the development of statistical properties of its solution, which is usually investigated by tracing a great amount of sample trajectories. Notice that the distribution of a random variable is able to  capture all of its possible statistical properties. However, distribution usually represents  an abstract measure and is inconvenient in applications. If the random variable possesses more smoothness, the  distribution may be absolutely continuous with respect to the Lebesgue measure, and hence the density function exists. In fact, all the possible statistical properties of a random variable can also be characterized by its density function. Moreover, if the density function possesses more smoothness, one can apply the tools of real analysis to investigate the statistic property, which widely enlarges the applications of probability theory. Specifically, we can reformulate the MVE as a system depending on density function once the solution of an  MVE has a density function. Moreover, the density function of an MVE satisfies a Fokker-Plank equation if it has more smoothness. Then based on the Fokker-Plank equation, we can independently investigate and approximate the corresponding density function. Once this is done, the MVE can be transformed to a normal SDE by virtue of density function, and then the dynamical behavior of sample trajectories can be easily simulated. 

The existence and regularity of the density functions of solutions to different kind of SDEs has been investigated by many scholars.
Nualart \cite{Nu2006} proved  the existence of a density function of the solution to an SDE  under H\"ormander condition and  global Lipschitz continuity. Furthermore, the density function exhibits smoothness if the coefficients have bounded infinite derivatives.
Antonelli and Kohatsu-Higa \cite{AnKo2002} investigated   a class of MVEs, where the  dependence on distribution is characterized by kernel functions. They proved that the   density function is infinitely differentiable  if the kernel functions have bounded derivatives of all order as well as  a restricted H\"ormander condition holds. 
Crisan and McMurray \cite{CrMc2018} established a new integration by parts formulae to MVE with the help of Lions' derivative with respect to distribution. Then, they investigated the a priori regularity estimate of the density function under the assumptions that the coefficients are sufficiently  differentiable and  Lions differentiable with respect to state variable and distribution, respectively. Additionally, they also assume all the derivatives are bounded and Lipschitz continuous as well as a uniform elliptic condition is fulfilled. Cui et al.\ \cite{CuHoSh2022} studied the existence of a density function of splitting AVF scheme for Langevin equation and its convergence  to the density function of the exact solution. Chen et al.\ \cite{ChCuHoSh2023} investigate the existence of density function of  exponential Euler scheme for stochastic heat equation and derive its convergence rate to the one of the exact solution.
Hong et al.\ \cite{HoJiSh2024} studied the convergence of the density function of a finite difference approximation to stochastic Cahn-Hilliard equation.

In this paper, we aim to study the existence and accurate order regularity of a density function of the solution to an MVE.  We first investigate the existence property only under  Lipschitz continuity and uniform elliptic conditions. In this case, the Lions derivatives of the  coefficients  do not exist, which results that we can not directly take Malliavin derivative on both side of  MVE to derive an equation that the Malliavin derivative of the solution satisfies.  Following the strategy for proving the  well-posedness of MVE  \cite{Nu2006}, we construct a sequence of approximate SDEs, whose solutions converge to the solution of MVE,  and a  sequence of measures, respectively. Then, we can conclude the existence of density functions of the solutions to SDEs \cite{Nu2006}, however, the convergence of solutions  does not guarantee the convergence of their  density functions. To address this issue,  we directly apply the Malliavin calculus to study the uniform a priori estimate and convergence of the Malliavin derivatives of the solutions to SDEs  and then deduce an estimate for the Malliavin derivative of the solution to MVE. Also, we manage to confirm the invertibility of Malliavin covariance matrix, which leads to the existence of the density function of solution to MVE. Furthermore, we also investigate the $N$-th order smoothness of density function to MVE  whenever  the coefficients are $(N+2)$-th differentiable with respect to the state variable.

The rest of this paper is organized as follows. In Section 2, we introduce some working spaces and  Malliavin calculus. In section 3, we establish the existence of the density function under Lipschitz continuity of coefficients.  In section 4, we study the precise order regularity of density function under finite smoothness assumption for coefficients. In section 5, we present several numerical experiments to illustrate the approximate density functions of solutions to MVEs via independently solving corresponding Fokker-Plank equations.

%%%%%%%%%%%%%%%%%%%%%%%%%%%%%%%%%%%%%%%%%%%%%%%%%%%%%%%%%%%%%%%%%%%%%%%%%%%%%%%%%%%%%%%%%%%%%%%%%%%%%%%%%%%%%%%%%%%%%%%%%%
\section{Preliminaries}
In this section we first describe some spaces and Wasserstein distance, then we introduce Malliavin calculus and some of its properties \cite{Ma1997,Nu2006, Sh2004}.

Let $\mathbb{R}^m$ be an $m$-dimensional Euclidean space with norm $ \| \cdot \|$ and inner product $(\cdot,\cdot)$. Let $T>0$, by $\mathbb{H} = L^2(0,T; \mathbb{R}^m)$ denote a  Hilbert space endowed with inner product $(h_1,h_2)_{\mathbb{H}} = \int_{0}^{T} (h_1(t),h_2(t)) dt$ and norm $ \| \cdot \|_{\mathbb{H}}$. If $m=1$, we write by $H = L^2(0,T;\mathbb{R})$. Define a tensor product space $\mathbb{H} \otimes \mathbb{H}$ as a set of bounded bilinear mappings from the product space $\mathbb{H} \times \mathbb{H} $ to $\mathbb{R}$, i.e.\ for any $ h_1 \otimes h_2 \in 
\mathbb{H} \otimes \mathbb{H}$ and $(u_1,u_2) \in \mathbb{H} \times \mathbb{H}$, there holds
$$  
h_1 \otimes h_2(u_1, u_2) = (h_1,u_1)_{\mathbb{H}} (h_2,u_2)_{\mathbb{H}}.
$$ 
For simplifying notations, we write $\mathbb{H}^{\otimes k} = \underbrace{\mathbb{H} \otimes \cdots \otimes \mathbb{H}}_{k}$  and $\mathbb{H}^{\times k} = \underbrace{\mathbb{H} \times \cdots \times \mathbb{H}}_k$ for $k \geq 2$. 

Let $(\Omega, \mathcal{F}, \mathbb{P})$ be a complete probability space and $W_t =(W_t^1,\dots,W_t^m)^T$ $(t\in[0,T])$ be a standard $m$-dimensional Brownian Motion, which generates a filtration $\mathcal{W}_t$. For any $\mathcal{F}$-measurable random variable $\xi$, denote by $\mathcal{L}[\xi]$ the law of $\xi$.
Denote by $L_d^p(\Omega)$ the space of $\mathbb{R}^d$-valued $\mathcal{F}$-measurable random variables $\xi$ satisfying $\mathbb{E}[\|\xi\|^p] < \infty $. As $d =1$, we write $L^p(\Omega) = L_1^p(\Omega)$. 

For any $a, b \in \mathbb{R}$, define  $a\vee b = max\{a, b\}$. Throughout this paper, we use $D$ to represent the Fr\'echet derivative  and $\mathcal{D}$ to denote the Malliavin derivative.

Denote by $C_b^n(\mathbb{R}^d)$ the space of functions defined on $\mathbb{R}^d$ that are $n$-times continuously differentiable and bounded along with all their derivatives up to order $n$. For any $f \in C_b^n(\mathbb{R}^d)$,  define a norm by
$ \|f \|_{C_b^n(\mathbb{R}^d)} =  \sum_{i=0}^{n} \sup_{x \in \mathbb{R}^d} \|D^if(x)\|_{(\mathbb{R}^{d})^{\otimes i}}$.

Let $\mathcal{P}(\mathbb{R}^d)$ be the space of probability measures on $\mathbb{R}^d$ and for any $p \geq 1$, denote by $\mathcal{P}_p(\mathbb{R}^d)$ the subspace of $\mathcal{P}(\mathbb{R}^d)$ with finite moment of order $p$,
\begin{equation*}
	\mathcal{P}_p(\mathbb{R}^d) = \{ \mu \in \mathcal{P}(\mathbb{R}^d): \int_{\mathbb{R}^d} \| x\|^p \mu(dx) < \infty \}.
\end{equation*}
Moreover, $\mathcal{P}_p(\mathbb{R}^d)$ is a Polish space under the $p$-Wasserstein distance
\begin{equation*}
	W_p(\mu,\nu) = \inf_{\pi \in \Pi(\mu,\nu)}(\int_{\mathbb{R}^d \times \mathbb{R}^d} \| x-y\|^p \pi(dx, dy))^{\frac{1}{p}} ,\quad  \mu,\nu \in \mathcal{P}_p(\mathbb{R}^d),
\end{equation*} 
where $\Pi(\mu,\nu)$ is the set of probability measures on $\mathbb{R}^d \times \mathbb{R}^d$ with two marginal measures $\mu$ and $\nu$, i.e.\ $\pi \in \Pi(\mu,\nu)$ 
implies  $\pi(\cdot \times \mathbb{R}^d) = \mu(\cdot)$ and $\pi(\mathbb{R}^d \times \cdot) = \nu(\cdot)$.
According to \cite[Section 2.2]{BaMeMe2020}, for any two $\mathbb{R}^d$-value random variables $X$ and $Y$ with finite $p$-order moments, there holds  
\begin{equation}\label{e1.2.1}
	W_p(\mathcal{L}[X],\mathcal{L}[Y])^p \leq \mathbb{E}[\| X-Y\|^p].
\end{equation} 

Now we start to introduce Malliavin Calculus. For any $h = (h^1,\dots, h^m)^T \in \mathbb{H}$, define $W(h) = \sum_{j=1}^{m} \int_{0}^{T} h^j(t) dW^j_t$ which is a Gaussian process on $\mathbb{H}$. For any integer $\alpha \geq 1$, denote by $C_p^{\infty}(\mathbb{R}^\alpha)$ the set of all infinitely continuously differentiable functions 
$f: \mathbb{R}^\alpha \to \mathbb{R}$ such that $f$ and all of its partial derivatives have at most polynomial growth. Define a set of smooth random variables by 
\begin{equation*}
	\mathcal{S} = \{ F = f(W(h_1), \dots, W(h_\alpha)) | f \in C_p^{\infty}(\mathbb{R}^\alpha),h_i \in \mathbb{H}, i = 1,\dots,\alpha, \alpha\geq 1\}. 
\end{equation*}  
The Malliavin derivative of $F \in \mathcal{S}$ is 
defined as an $\mathbb{H}$-valued random variable given by
\begin{equation*}
	\mathcal{D}F = \sum_{i=1}^\alpha \partial_if(W(h_1),\dots,W(h_\alpha))h_i =\sum_{j=1}^m\sum_{i=1}^\alpha \partial_if(W(h_1),\dots,W(h_\alpha)h_i^je_j,
\end{equation*}
where $h_i(t) = \sum_{j=1}^m h_i^j(t)e_j$, $h^j_i \in H$ and $\{e_j\}_{j=1}^m$ is the standard basis of $\mathbb{R}^m$.
Obviously, $\mathcal{D}F=\{\mathcal{D}_tF : t\in[0,T]\}$ can be regarded as an $\mathbb{R}^m$-valued stochastic process, 
whose $j$-th component reads   
\begin{equation*}
	\mathcal{D}_t^jF = \sum_{i=1}^\alpha \partial_if(W(h_1),\dots,W(h_\alpha)h_i^j(t).
\end{equation*} 

By induction, for any integer $n \geq 2$ and $F \in \mathcal{S}$, the $n$-th order Malliavin derivative $\mathcal{D}^{n}F$ is defined as a bounded symmetric multilinear operator from $\mathbb{H}^{\times n}$ to $\mathbb{R}$ given by
\begin{equation}\label{e1.2.1-1}
	\mathcal{D}^{n}F = \sum_{i_1,\dots i_n= 1}^{\alpha} \partial_{i_1}\cdots\partial_{i_{n}}f(W(h_1),\dots,W(h_\alpha)) h_{i_1} \otimes \cdots \otimes h_{i_{n}}. 
\end{equation} 
Moreover, $\mathcal{D}^{n} F= \{\mathcal{D}_{t_1,\dots,t_{n}}F : (t_1, \dots, t_{n})\in [0,T]^{\times n}\}$ is an $(\mathbb{R}^m)^{\otimes n}$-valued stochastic process of the form
\begin{equation}\label{e1.2.1-2}
	\begin{aligned}
		&\mathcal{D}_{t_1,\dots,t_{n}}F = \\
		&\sum_{j_1, \dots, j_n=1}^{m}\sum_{i_1,\dots, i_n=1}^{\alpha} \partial_{i_1}\cdots\partial_{i_{n}}f(W(h_1),\dots,W(h_\alpha)) h_{i_1}^{j_1}(t_1) \cdots h_{i_{n}}^{j_{n}}(t_{n}) e_{j_1}\otimes \cdots \otimes e_{j_{n}}.
	\end{aligned}
\end{equation}
For any given $1 \leq j_1,\dots,j_{n} \leq m$, the  component associated with index $(j_1,\dots,j_{n})$ of the $n$-th order Malliavin derivative $\mathcal{D}^{n}F$ is denoted by
\begin{equation}\label{e1.2.1-3}
	\mathcal{D}_{t_1,\dots,t_{n}}^{j_1,\dots,j_{n}}F = \sum_{i_1,\dots,i_n=1}^{\alpha} \partial_{i_1}\cdots\partial_{i_{n}}f(W(h_1),\dots,W(h_\alpha)) h_{i_1}^{j_1}(t_1) \cdots h_{i_{n}}^{j_{n}}(t_{n}). 
\end{equation}

For any $n, p \geq 1$, define a Watanabe-Sobolev space $\mathbb{D}^{n,p}$ as the closure of the class of smooth random variables $\mathcal{S}$ with respect to a norm
\begin{equation*}
	\|  F \|_{n,p} = \big(\mathbb{E}\big[|F|^p + \sum_{k=1}^{n}\|  \mathcal{D}^kF \|^p_{\mathbb{H}^{\otimes k}}\big]\big)^{\frac{1}{p}} < \infty.
\end{equation*}
Similarly, let $(V,\|\cdot\|_V)$ be a real separable Hilbert space and we define a space $\mathbb{D}^{n,p}(V)$ as the completion of $V$-valued smooth random variables with respect to a norm
\begin{equation*}
	\|  F \|_{n,p,V} = \big(\mathbb{E}\big[\| F\|_V^p + \sum_{k=1}^{n}\|  \mathcal{D}^kF \|^p_{\mathbb{H}^{\otimes k}\otimes V}\big]\big)^{\frac{1}{p}} < \infty.
\end{equation*}

The next lemma provides a way to estimate the Malliavin derivative.
\begin{lemma}\label{L1.2.1-1}
	Let $F$ be a $V$-valued smooth random variable. Assume that for any given $p \geq 2$ and integer $n \geq 1$, there exists a constant $L>0$ such that $\mathbb{E}[\|\mathcal{D}_{t_1,\dots,t_n}^{j_1,\dots,j_n}F\|^p_{V}] \leq L $ for all $t_1,\dots,t_n \in [0,T]$ and $1 \leq j_1,\dots,j_n \leq m$. Then, there exists a constant $C=C(T,p,n,L)>0$ such that $\mathbb{E}[\|\mathcal{D}^nF\|^p_{\mathbb{H}^{\otimes n}\otimes V}] \leq C$.  
\end{lemma}	

\begin{proof}
	Let $\vec{t} = (t_1,\dots, t_n)$. For any $g_i = (g_i^1,\dots,g_i^m)^T\in \mathbb{H}$ with $g_i^j \in H$, $1\leq i \leq n$ and $1\leq j \leq m$, noticing \eqref{e1.2.1-1}--\eqref{e1.2.1-3} and H\"older inequality, we have 
	\begin{equation*}
		\begin{aligned}
			&\ \|\mathcal{D}^nF(g_1,\dots,g_n)\|_V \\
			=&\ \|\int_{[0,T]^{\times n}}\mathcal{D}_{t_1,\dots,t_n}F(g_1(t_1), \dots, g_n(t_n)) d\vec{t}\ \|_V\\
			\leq&\ \int_{[0,T]^{\times n}}\| \sum_{j_1,\dots,j_n=1}^{m} \mathcal{D}_{t_1,\dots,t_n}^{j_1,\dots,j_n}F e_{j_1}\otimes \cdots \otimes e_{j_n} (g_1(t_1),\dots,g_n(t_n))\|_V d\vec{t}\\
			= &\ \int_{[0,T]^{\times n}} \|\sum_{j_1,\dots,j_n=1}^{m}\mathcal{D}_{t_1,\dots,t_n}^{j_1,\dots,j_n}Fg_1^{j_1}(t_1)\cdots g_n^{j_n}(t_n)\|_V d\vec{t}\\
			\leq &\  \int_{[0,T]^{\times n}}\sum_{j_1,\dots,j_n=1}^{m} \|\mathcal{D}_{t_1,\dots,t_n}^{j_1,\dots,j_n}Fg_1^{j_1}(t_1) \cdots g_n^{j_n}(t_n)\|_V d\vec{t}\\
			= &\ \sum_{j_1,\dots,j_n=1}^{m} \int_{[0,T]^{\times n}} \|\mathcal{D}_{t_1,\dots,t_n}^{j_1,\dots,j_n}F\|_V |g_1^{j_1}(t_1)| \cdots |g_n^{j_n}(t_n)| d\vec{t}\\
			\leq &\ \sum_{j_1,\dots,j_n=1}^{m} (\int_{[0,T]^{\times n}} \|\mathcal{D}_{t_1,\dots,t_n}^{j_1,\dots,j_n}F\|^2_V d\vec{t}\ )^{\frac{1}{2}} (\int_{0}^{T}|g_1^{j_1}(t_1)|^2 dt_1 \cdots \int_{0}^{T}|g_n^{j_n}(t_n)|^2 dt_n)^{\frac{1}{2}}\\
			= &\ \sum_{j_1,\dots,j_n=1}^{m} (\int_{[0,T]^{\times n}} \|\mathcal{D}_{t_1,\dots,t_n}^{j_1,\dots,j_n}F\|^2_V d\vec{t}\ )^{\frac{1}{2}} \|g_1^{j_1}\|_H\cdots\|g_n^{j_n}\|_H.
		\end{aligned}
	\end{equation*}
	By H\"older inequality, we obtain that 
	\begin{equation*}
		\begin{aligned}
			&\ \mathbb{E}[\|\mathcal{D}^nF(g_1,\dots,g_n)\|^p_V] \\
			\leq &\   m^{n(p-1)}\mathbb{E}[ \sum_{j_1,\dots,j_n=1}^{m} (\int_{[0,T]^{\times n}} \|\mathcal{D}_{t_1,\dots,t_n}^{j_1,\dots,j_n}F\|^2_V d\vec{t}\ )^{\frac{p}{2}} \|g_1^{j_1}\|^p_H\cdots\|g_n^{j_n}\|^p_H]\\
			\leq &\  m^{n(p-1)} \sum_{j_1,\dots,j_n=1}^{m}  \mathbb{E}[T^{\frac{n(p-2)}{2}}\int_{[0,T]^{\times n}}  \|\mathcal{D}_{t_1,\dots,t_n}^{j_1,\dots,j_n}F\|^p_V d\vec{t}\ ] \cdot \|g_1^{j_1}\|^p_H\cdots\|g_n^{j_n}\|^p_H\\
			\leq &\ m^{n(p-1)}T^{\frac{n(p-2)}{2}}L^pT^n \sum_{j_1,\dots,j_n=1}^{m}\|g_1^{j_1}\|^p_H\cdots\|g_n^{j_n}\|^p_H \\
			\leq &\ m^{n(p-1)}T^{\frac{np}{2}}L^p m^n\|g_1\|^p_{\mathbb{H}}\cdots\|g_n\|^p_{\mathbb{H}},
		\end{aligned}
	\end{equation*}
	which completes the proof. 	 
\end{proof}

For any $F=(F^1,\dots,F^d)^T \in \mathbb{D}^{n,p}(\mathbb{R}^d)$, its Malliavin covariance matrix $Q = (Q_{i,l})$ is defined as $Q_{i,l} = (\mathcal{D}F^i,\mathcal{D}F^l)_{\mathbb{H}}$ for $1 \leq i, l \leq d$.
Define some spaces 
\begin{equation*}
	L^{\infty-}(\Omega) = \cap_{p \geq 1} L^p(\Omega),\qquad
	\mathbb{D}^{n,\infty} = \cap_{p \geq 1} \mathbb{D}^{n,p},\qquad 
	\mathbb{D}^{n,\infty}(V) = \cap_{p \geq 1} \mathbb{D}^{n,p}(V).	
\end{equation*}

The following two lemmas provide  criteria for determining the existence and smoothness of density function to a random variable. 

\begin{lemma}({\cite[Theorem 2.1.2]{Nu2006}})\label{L1.2.1}
	Let $p > 1$, $G \in  \mathbb{D}^{1,p}_{loc}(\mathbb{R}^d)$ and its Malliavin covariance matrix $Q$ is invertible a.s.
	Then the law of $G$ is absolutely continuous with respect to the Lebesgue measure on $\mathbb{R}^d$.
\end{lemma}

\begin{lemma}(\cite[Theorem 5.9]{Sh2004})\label{L1.2.2}
	Let $G$ be an $\mathbb{R}^d$-valued random variable with density function $f$ and $\triangle = det Q $. Take $p > d$ and for any integer $N \geq 0$ , let $ M = \frac{(N+1)^2}{2} + \frac{3(N+1)}{2}$. 
	If $G$ belongs to the space $\mathbb{D}^{N+2, 4dMp}(\mathbb{R}^d)$ and $\triangle^{-1}$ belongs to the space $L^{2Mp}(\Omega)$. 
	Then it follows that $f \in C_b^{N}(\mathbb{R}^d)$ and
	\begin{equation*}
		\|  f \|_{C_{b}^{N}(\mathbb{R}^d)} \leq C(1 + \|  \mathcal{D}G  \|_{N+1,4dMp,\mathbb{R}^d}^{2dM} \|  \triangle^{-1} \|_{L^{2Mp}(\Omega)}^{M})
		\|  f \|_{C_{b}(\mathbb{R}^d)}^{1-\frac{1}{p}},
	\end{equation*}
	where $C = C(p,d,N) >0$ is a constant.  
\end{lemma}

The next lemma  extends the chain rule for Malliavin derivative to the case of Lipschitz continuity.

\begin{lemma}(\cite[Proposition 1.2.4]{Nu2006})\label{L1.2.2-1}
	Let $\phi : \mathbb{R}^d \rightarrow \mathbb{R}$ be a function such that $\| \phi(x)- \phi(y)\| \leq L\| x-y\|$ for $x,y \in \mathbb{R}^d$.
	Suppose that $F \in \mathbb{D}^{1,2}(\mathbb{R}^d)$, then $\phi(F) \in \mathbb{D}^{1,2}$ and there exists a random vector $G = (G_1, \dots, G_d)$ bounded by $L$ such that $\mathcal{D}(\phi(F)) = \sum_{i=1}^{d} G_i\mathcal{D}F^i$.
\end{lemma}

The next lemma  provides a technique to study  the invertibility of the Malliavin covariance matrix. 

\begin{lemma}(\cite[Lemma 2.3.1]{Nu2006})\label{L1.2.3}
	Let $Q$ be a symmetric nonnegative definite $d \times d$ random matrix. Assume that the entries $Q_{i,l}$ have moments of all orders and that
	for any $p \geq 2$ there exists $\epsilon_0(p)$ such that for all $\epsilon \leq \epsilon_0(p)$ there holds
	\begin{equation*}
		\sup_{\| v\|=1} \mathbb{P}\{v^T Q v \leq \epsilon \} \leq \epsilon^p.
	\end{equation*} 
	Then $(det Q)^{-1} \in L^p(\Omega)$.
\end{lemma}

The Burkholder-Davis-Gundy's (BDG) inequality plays an important role in stochastic analysis associated with It\'o integral.

\begin{lemma}(Burkholder-Davis-Gundy)\cite[Theorem 7.3]{Ma2008}\label{L1.2.4} %\cite{Bu1966, Da1970, BuGu1970}
	For any $p > 0$, there exist two positive constants $c_p < C_p$ such that for all $d$-dimensional martingales $X(t) = \int_{0}^{t} f(s) dW_s$ $(t \in [0,T])$ with $f \in  \mathbb{L}^2_{d\times m}(0,T)$, there holds 
	\begin{equation*}
		c_p\mathbb{E}[(\int_{0}^{t} \|f(s)\|^2 ds)^{\frac{p}{2}}] \leq \mathbb{E}[\sup_{0 \leq \tau \leq t} \|\int_{0}^{\tau} f(s) dW_s\|^p] \leq 	C_p\mathbb{E}[(\int_{0}^{t} \|f(s)\|^2 ds)^{\frac{p}{2}}].  
	\end{equation*} 
\end{lemma}

%%%%%%%%%%%%%%%%%%%%%%%%%%%%%%%%%%%%%%%%%%%%%%%%%%%%%%%%%%%%%%%%%%%%%%%%%%%%%%%%%%%%%%%%%%%%%%%%%%%%%%%%%%%%%%%%%%%%%%%%%%

\section{Existence of density function}
In this section, we will study the existence of density function to the solution of MVE under Lipschitz continuity of coefficients. Consider a general Mckean-Vlasov equation 
\begin{equation*}\label{e1.1.1}
	\begin{aligned}
		dX(t)&= b(t,X(t), \mathcal{L}[X(t)]) dt + \sigma(t, X(t), \mathcal{L}[X(t)]) dW_t, \\
		X(0) &= \xi,
	\end{aligned}
\end{equation*}
where $b$ and $\sigma$ are the drift and diffusion coefficients, $\xi$ is an $\mathcal{F}$-measurable random variable. Its integral form is given by

\begin{equation}\label{e1.3.1}
	X(t) = \xi + \int_{0}^{t}b(s,X(s),\mathcal{L}[X(s)]) ds+ \int_{0}^{t}\sigma(s, X(s), \mathcal{L}[X(s)]) dW_s.
\end{equation}
We assume for some $p\geq 2$,
\begin{itemize}
	\item[(H1)] $\xi \in L_d^p(\Omega)$ is independent of $W_t$ for all $t >0$.
	\item[(H2)]
	$b: [0,T]\times \mathbb{R}^d \times \mathcal{P}_2(\mathbb{R}^d) \rightarrow \mathbb{R}^d$ and $ \sigma: [0,T]\times \mathbb{R}^d \times \mathcal{P}_2(\mathbb{R}^{d}) \rightarrow \mathbb{R}^{d\times m}$ satisfy uniform Lipschitz condition, i.e.\ there exists $L>0$ such that for all $t \in [0,T]$, $x,x' \in \mathbb{R}^d$ and $\mu,\mu' \in \mathcal{P}_2(\mathbb{R}^d)$
	\begin{equation*}
		\begin{aligned}
			\| b(t,x,\mu) - b(t, x', \mu')\| \leq L(\| x-x'\| + W_2(\mu,\mu')) ,\\
			\|  \sigma(t,x,\mu) - \sigma(t, x', \mu') \| \leq L(\| x-x'\| + W_2(\mu,\mu')).
		\end{aligned}
	\end{equation*}
\end{itemize}

Define a matrix value function $A(t,x,\mu) = \sigma(t,x,\mu)\sigma^T(t,x,\mu)$.

\begin{itemize}
	\item[(H3)] There exists $\lambda > 0$ such that $v^TA(t,x,\mu)v \geq \lambda\| v\|^2$ for all $v \in \mathbb{R}^d$ and $(t,x,\mu) \in [0,T] \times \mathbb{R}^d \times \mathcal{P}_2(\mathbb{R}^d)$. 
\end{itemize}   

\begin{remark}\label{L1.3.1}
	1. From (H2) it follows that $b$ and $\sigma$ are of linear growth, i.e.\ for any  $t \in [0,T]$, $x \in \mathbb{R}^d$, $\mu \in \mathcal{P}_2(\mathbb{R}^d)$, $b$ and $\sigma$, there holds 
	\begin{equation*}
		\| b(t,x,\mu)\| \vee \|  \sigma(t,x,\mu) \|\leq C(1 +\|x\|+ W_2(\mu,\delta_0)),
	\end{equation*}
	where $\delta_0$ is the  Dirac measure on $\mathbb{R}^d$ centered at origin.
	
	2. (H3) implies that $A(t,x,\mu)$ is a symmetric,  strictly positive definite matrix with bounded inverse. 
	
	3. Let $\sigma = (\sigma^1, \dots, \sigma^m)$ $(\sigma^k \in \mathbb{R}^d, k=1,\dots,m)$, then  \eqref{e1.3.1} can be equivalently rewritten as
	\begin{equation}\label{e1.3.2}
		X(t) = \xi + \int_{0}^{t}b(s,X(s),\mathcal{L}[X(s)]) ds+ \sum_{k=1}^m\int_{0}^{t}\sigma^k(s, X(s), \mathcal{L}[X(s)]) dW_s^k.
	\end{equation}
\end{remark}

In order to simplify notations, we will denote by $b(s) = b(s,X(s),\mathcal{L}[X(s)])$ and $\sigma^k(s) = \sigma^k(s,X(s),\mathcal{L}[X(s)])$ in forthcoming analysis if no confusion occurs.

Define a filtration $\mathcal{F}_t = \sigma\{\xi, W(s), s\in[0,t]\}$. By $\mathbb{L}^p_d(0,T)$ we denote the set of $\mathcal{F}_t$-progressively measurable  $\mathbb{R}^d$-valued processes $X(t)$  satisfying $\|X(t)\|^p_{\mathbb{L}^p_d}=\mathbb{E}[\int_{0}^{T}\| X(t)\|^p dt] < \infty $. As $d=1$, we write $ \mathbb{L}^p(0,T) :=\mathbb{L}^p_1(0,T)$. According to \cite[Theorem 1.7]{Ca2016}, under (H1) and (H2), \eqref{e1.3.1} possesses a unique solution $X(t) \in \mathbb{L}^2_d(0,T)$. Moreover, we  have extra regularity estimate for $X(t)$. 

\begin{lemma}\label{L1.3.2}
	Assume (H1), (H2) and $p \geq 2$. Let $X(t)$ be a solution to \eqref{e1.3.1}. Then,
	\begin{equation*}
		\mathbb{E}[ \| X(t)\|^p] \leq C(1+ \mathbb{E}[\|\xi\|^p]), \quad \forall t\in[0,T],
	\end{equation*} 
	where $C =C(T,p) > 0$ is a constant.
\end{lemma}

\begin{proof}
	First, we study  estimates for $b(t)$ and $\sigma(t)$. In terms of \eqref{e1.2.1} and Remark \ref{L1.3.1}, direct computation gives
	\begin{equation*}
		\begin{aligned}
			\mathbb{E}[\|b(t)\|^p] &\ \leq \mathbb{E}[ C^p(1  + \| X(t)\| + W_2(\mathcal{L}[X(t)], \delta_0))^p]\\
			&\ \leq 3^{p-1}C^p\mathbb{E}[ 1  + \| X(t)\|^p + (W_2(\mathcal{L}[X(t)], \delta_0))^p]\\
			&\ \leq C(p)\mathbb{E}[ 1  + \| X(t)\|^p + \mathbb{E}[\|X(t)\|^2]^{\frac p2}]\\
			&\ \leq C(1 + \mathbb{E}[\|X(t) \|^p]),
		\end{aligned}
	\end{equation*}
	where $C = C(p)>0$ is a constant.
	
	In a similar way, we obtain $\mathbb{E}[\|\sigma(t)\|^p] \leq C(1 + \mathbb{E}[\|X(t) \|^p])$.
	Then applying BDG inequality and  H\"older inequality, we have
	\begin{align*}
		&\mathbb{E}[\|X(t)\|^p] \\
		\leq&3^{p-1}(\mathbb{E}[\|\xi\|^p] + \mathbb{E}[\|\int_{0}^{t}b(s) ds\|^p]
		+\mathbb{E}[\|\int_{0}^{t}\sigma(s) dW_s\|^p])\\
		\leq&3^{p-1}(\mathbb{E}[\|\xi\|^p] +\mathbb{E}[(\int_{0}^{t}\|b(s)\|ds)^p] 
		+ \mathbb{E} [\sup_{0 \leq \tau \leq t} \|\int_{0}^{\tau} \sigma(s) dW_{s}\|^p])\\
		\leq&3^{p-1}\mathbb{E}[\|\xi\|^p] +3^{p-1}T^{p-1}\mathbb{E}[\int_{0}^{t}\|b(s)\|^p ds] 
		+3^{p-1}C(p)\mathbb{E} [(\int_{0}^{t} \|\sigma(s)\|^2 ds)^{\frac{p}{2}}]\\
		\leq&C\mathbb{E}[\|\xi\|^p] +  C\mathbb{E}[\int_{0}^{t} \|b(s)\|^p  ds] + CT^{\frac{p-2}{2}}\mathbb{E} [\int_{0}^{t}\| \sigma(s)\|^p ds]\\
		\leq&C(1+\mathbb{E}[\|\xi\|^p]) +C\int_{0}^{t}\mathbb{E}[\|X(s)\|^p] ds,
	\end{align*}
	where $C = C(T,p) > 0$ is a constant.    
	By Gronwall's inequality, we obtain
	\begin{equation*}
		\mathbb{E}[\| X(t)\|^p] \leq C(1+ \mathbb{E}[\|\xi\|^p])e^{Ct} \leq C(1+ \mathbb{E}[\|\xi\|^p]),
	\end{equation*}
	which completes the proof.  
\end{proof}

The next lemma  provides a technique for verifying  Malliavin differentiability of a random variable $F$.

\begin{lemma}\cite[Lemma 1.2.3]{Nu2006}\label{L1.3.3}
	Let $\{F_n, n \geq 1\}$ be a sequence of random variables in $\mathbb{D}^{1,2}$ that converges to $F$ in $L^2(\Omega)$ and such that
	$\sup_{n} \mathbb{E}[\| \mathcal{D}F_n \|_{\mathbb{H}}^2] < \infty$. Then $F$ belongs to $\mathbb{D}^{1,2}$ and the sequence of derivatives
	$\{\mathcal{D}F_n, n \geq 1\}$ converges to $\mathcal{D}F$ in the weak topology of $L^2(\Omega; \mathbb{H})$.
\end{lemma}

\begin{remark}\label{L1.3.4}
	We can easily extend this lemma to more general case.   
	For any $p \geq 2$, assume that $ F_n  \in \mathbb{D}^{1,p}$ converges to $F$ in $L^p(\Omega)$ and $\mathbb{E}[\|\mathcal{D}F_n\|^p_{\mathbb{H}}] \leq C$ for $n \geq 1$. According to this lemma, $\mathcal{D}F_n $ converges to $\mathcal{D}F \in L^2(\Omega; \mathbb{H}) $ in the weak topology of $L^2(\Omega; \mathbb{H})$. Let $F_{n(k)}$ be any subsequence of $F_n$, then $\mathbb{E}[\|\mathcal{D}F_{n(k)}\|^p_{\mathbb{H}}] \leq C$ $(k \geq 1)$ and there exists a subsequence of Malliavin derivatives $\mathcal{D}F_{n(k)}$ such that it converges to $\beta \in L^p(\Omega; \mathbb{H})\subset L^2(\Omega; \mathbb{H})$ in the weak topology of $L^p(\Omega; \mathbb{H})$. The uniqueness of limit implies $\beta = \mathcal{D}F$. Thus, $F \in \mathbb{D}^{1,p}$ and  $\mathcal{D}F_n$ converges to $\mathcal{D}F$ in the weak topology of $L^p(\Omega; \mathbb{H})$.
\end{remark}

In order to study a priori estimate for the Malliavin derivative of solution $X(t)$, we construct two sequences $X_n \in \mathbb{L}^2_d(0,T)$ and $\mu^n \in C([0,T]; \mathcal{P}_2(\mathbb{R}^d))$ as follows. Take $X_0 = \xi$ and $\mu^0=(\mu^0_t)_{0\leq t\leq T} = \mathcal{L}[\xi]$.  By induction, for any $n\geq 1$, define $X_n$ as the solution to a stochastic differential equation
\begin{equation}\label{e1.3.3}
	X_n(t) = \xi + \int_{0}^{t}b(s,X_n(s), \mu^{n-1}_s) ds + \sum_{k=1}^{m} \int_{0}^{t}\sigma^k(s, X_n(s), \mu^{n-1}_s) dW_s^k, \quad t\in [0,T].
\end{equation}
Then define $\mu^n = (\mu^n_t)_{0\leq t \leq T} =(\mathcal{L}[X_{n}(t)])_{0\leq t \leq T}$. According to \cite[Theorem 1.7]{Ca2016} and its proof, $X_n(t)$ converges to the solution $X(t)$ of \eqref{e1.3.1} in $\mathbb{L}^2_d(0,T)$ and $\mu^n$ tends to $\mu = (\mu_t)_{0\leq t \leq T} = (\mathcal{L}[X(t)])_{0\leq t\leq T}$ in $C([0,T]; \mathcal{P}_2(\mathbb{R}^d))$  as $n \to \infty$, which means 
\begin{equation}\label{e1.3.3-1}
	\lim_{n \to \infty} \sup_{t\in[0,T]} W_2(\mu^n_t, \mathcal{L}[X(t)]) = 0.
\end{equation}
Moreover, we investigate the convergence of $X_n(t)$  to  $X(t)$ in the $L^p_d(\Omega)$ space.

\begin{lemma}\label{L1.3.5}
	Assume (H1) and (H2). Let $X(t)$ be a solution to \eqref{e1.3.2}. Then,  
	\begin{equation}\label{e1.3.4}
		\lim_{n \to \infty}	\mathbb{E}[\|X_n(t) - X(t)\|^p] = 0, \quad \forall t \in [0,T]. 
	\end{equation} 
\end{lemma}

\begin{proof}
	By BDG inequality and H\"older inequality we have
	\begin{align*}
		&\mathbb{E}[ \| X_n(t) - X(t)\|^p]\\
		\leq& (m+1)^{p-1} \mathbb{E}[\| \int_{0}^{t} b(s,X_n(s),\mu^{n-1}_{s}) - b(s,X(s),\mathcal{L}[X(s)]) ds\|^p] \\
		&+(m+1)^{p-1} \sum_{k=1}^{m}\mathbb{E}[\| \int_{0}^{t} \sigma^k(s,X_n(s),\mu^{n-1}_{s}) - \sigma^k(s,X(s),\mathcal{L}[X(s)]) dW_{s}^k\|^p]\\
		\leq&C \mathbb{E}[ T^{p-1}\int_{0}^{t} \| b(s,X_n(s),\mu^{n-1}_{s}) - b(s,X(s),\mathcal{L}[X(s)])\|^p ds] \\
		&+C\sum_{k=1}^{m} \mathbb{E}[\sup_{0\leq \tau \leq t}\| \int_{0}^{\tau} \sigma^k(s,X_n(s),\mu^{n-1}_{s}) - \sigma^k(s,X(s),\mathcal{L}[X(s)]) dW_{s}^k\|^p]\\
		\leq&C\mathbb{E}[ \int_{0}^{t} 2^{p-1}L^p(\| X_n(s) - X(s)\|^p+ (W_2(\mu^{n-1}_{s} ,\mathcal{L}[X(s)]))^p) ds] \\
		&+C\sum_{k=1}^{m} C(p)\mathbb{E}[(\int_{0}^{t} \|\sigma^k(s,X_n(s),\mu^{n-1}_{s}) - \sigma^k(s,X(s),\mathcal{L}[X(s)])\|^2 ds)^{\frac{p}{2}}]\\
		\leq&C\mathbb{E}[ \int_{0}^{t} \| X_n(s) - X(s)\|^p+ (W_2(\mu^{n-1}_{s} ,\mathcal{L}[X(s)]))^p ds] \\
		&+C\sum_{k=1}^{m} T^{\frac{p-2}{2}} \mathbb{E}[ \int_{0}^{t} \|\sigma^k(s,X_n(s),\mu^{n-1}_{s}) - \sigma^k(s,X(s),\mathcal{L}[X(s)])\|^p ds]\\
		\leq&C\int_{0}^{t} \mathbb{E}[\| X_n(s) - X(s)\|^p] ds + C\int_{0}^{t}(W_2(\mu^{n-1}_{s} ,\mathcal{L}[X(s)]))^p ds.
	\end{align*}
	By Gronwall's inequality and \eqref{e1.3.3-1}, we get
	\begin{equation*}
		\begin{aligned}
			\mathbb{E}[ \| X_n(t) - X(t)\|^p] \leq& Ce^{Ct}\int_{0}^{t}(W_2(\mu^{n-1}_{s},\mathcal{L}[X(s)]))^p ds\\
			\leq& CTe^{CT} (\sup_{s \in [0,T]}W_2(\mu^{n-1}_{s},\mathcal{L}[X(s)]))^p \to 0, \text{ as } n\to \infty,
		\end{aligned}
	\end{equation*}
	which completes the proof. 
\end{proof}

According to \cite[Theorem 2.2.1]{Nu2006}, for any $1\leq j \leq m$, $n \geq 0$ and $t \in [0,T]$, $X_n(t) \in \mathbb{D}^{1,2}(\mathbb{R}^d)$ and its $j$-th component $\mathcal{D}_r^jX_n(t)$ $(r \in [0,T])$ of the Malliavin derivative $\mathcal{D}X_n(t)$ satisfies $\mathcal{D}_r^jX_n(t) = \mathcal{D}_r^jX_n(t)\mathbbm{1}_{\{r\leq t\}}$, where $\mathbbm{1}_{\{r\leq t\}}$ is an indicator function.
Applying Lemma \ref{L1.2.2-1}, we deduce that $b(t,X_n(t), \mu_t^{n-1})$ and $\sigma^k(t, X_n(t), \mu^{n-1}_t)$ $(1\leq k \leq m)$ belong to $\mathbb{D}^{1,2}(\mathbb{R}^d)$ and there exist $\mathcal{F}_t$-progressively measurable $\mathbb{R}^{d\times d}$-valued processes  $\bar{b}(t,\mu_t^{n-1})$ and $\bar{\sigma}^k(t,\mu_t^{n-1})$ uniformly bounded by $L$, such that
\begin{equation*}
	\begin{aligned}
		\mathcal{D}_r^j[b(t,X_n(t), \mu_t^{n-1})] &= \bar{b}(t, \mu^{n-1}_t)\mathcal{D}_r^jX_n(t)\mathbbm{1}_{\{r\leq t\}},\\
		\mathcal{D}_r^j[\sigma^k(t, X_n(t), \mu^{n-1}_t)] &=\bar{\sigma}^{k}(t, \mu^{n-1}_t) \mathcal{D}_r^jX_n(t)\mathbbm{1}_{\{r\leq t\}},
	\end{aligned}
\end{equation*}
Obviously $\mathcal{D}_r^jX(t) = 0$ for $0 \leq t < r \leq T$. And for $0 \leq r \leq t \leq T$, taking Malliavin derivative on both side of \eqref{e1.3.3} yeilds an SDE
\begin{equation*}\label{e1.3.5}
	\begin{aligned}
		\mathcal{D}_r^jX_n(t) =		&\ \int_{r}^{t} \bar{b}(s, \mu^{n-1}_s)\mathcal{D}_r^jX_n(s) ds \\ 
		&\ +\sigma^j(r, X_n(r), \mu^{n-1}_r) +\sum_{k=1}^{m}\int_{r}^{t} \bar{\sigma}^{k}(s, \mu^{n-1}_s) \mathcal{D}_r^jX_n(s) dW_s^k
	\end{aligned}
\end{equation*}  
with initial value   $\mathcal{D}_r^jX_n(r) = \sigma^j(r, X_n(r), \mu^{n-1}_r)$.
According to \eqref{e1.2.1}, Remark \ref{L1.3.1}, Lemma \ref{L1.3.2} and H\"older inequality,  we can estimate the initial value 
\begin{equation}\label{e1.3.6}
	\begin{aligned}
		\mathbb{E}[\| \sigma^j(r,X_{n}(r),\mu^{n-1}_{r})\|^p]
		\leq&\ C^p\mathbb{E}[(1+ \|X_{n}(r)\| + W_2(\mu_r^{n-1}, \delta_0))^p]\\
		\leq&\ 3^{p-1}C^p(1+\mathbb{E}[\|X_n(r)\|^p] + (W_2(\mu_r^{n-1}, \delta_0))^p)\\
		\leq&\ C(1+C(1 + \mathbb{E}[\|\xi\|^p])+	(\mathbb{E}[\|X_{n-1}(r)\|^2])^{\frac{p}{2}})\\
		\leq&\ C(1+\mathbb{E}[\|\xi\|^p]+	\mathbb{E}[\|X_{n-1}(r)\|^p])\\
		\leq&\ C(1+\mathbb{E}[\|\xi\|^p]),
	\end{aligned}
\end{equation}
where $C = C(T,p) > 0$ is a constant.

Now, we investigate the estimate of the  Malliavin derivative $\mathcal{D}X_n(t)$.  Applying BDG inequality, H\"older inequality and \eqref{e1.3.6}, we have 
\begin{align*}
	&\ \mathbb{E}[\| \mathcal{D}_r^jX_n(t)\|^p]\\
	\leq&\ (m+2)^{p-1}\mathbb{E}[\| \int_{r}^{t}\bar{b}(s, \mu_{s}^{n-1})\mathcal{D}_r^jX_n(s) ds\|^p \\
	&\ + \| \sigma^j(r, X_n(r), \mu_r^{n-1})\|^p   + \sum_{k=1}^{m} 
	\| \int_{r}^{t}\bar{\sigma}^{k}(s, \mu_{s}^{n-1})\mathcal{D}_r^jX_n(s) dW_{s}^k\|^p]\\
	\leq&\ C\mathbb{E}[(\int_{r}^{t}\|\bar{b}(s, \mu_{s}^{n-1})\mathcal{D}_r^jX_n(s)\|ds)^p] \\
	&\ + C(1+\mathbb{E}[\|\xi\|^p]) + C\sum_{k=1}^{m} 
	\mathbb{E}[(\sup_{r \leq \tau \leq t}\| \int_{r}^{\tau}\bar{\sigma}^{k}(s, \mu_{s}^{n-1})\mathcal{D}_r^jX_n(s) dW_{s}^k\|)^p]\\
	\leq&\ C\mathbb{E}[T^{p-1}\int_{r}^{t}\| \bar{b}(s, \mu_{s}^{n-1})\mathcal{D}_r^jX_n(s)\|^p ds]\\
	&\ + C(1+\mathbb{E}[\|\xi\|^p]) + C\sum_{k=1}^{m} C(p) \mathbb{E}[(\int_{r}^{t}\|\bar{\sigma}^{k}(s, \mu_{s}^{n-1})\mathcal{D}_r^jX_n(s)\|^2 ds)^{\frac{p}{2}}] \\
	\leq&\ C\mathbb{E}[\int_{r}^{t}L^p\|\mathcal{D}_r^jX_n(s)\|^p ds] \\
	&\ + C(1+\mathbb{E}[\|\xi\|^p]) + C\sum_{k=1}^{m} T^{\frac{p-2}{2}}\mathbb{E}[\int_{r}^{t}\|\bar{\sigma}^{k}(s, \mu_{s}^{n-1})\mathcal{D}_r^jX_n(s)\|^p ds]\\
	\leq&\ C(1+\mathbb{E}[\|\xi\|^p]) + C\int_{0}^{t} \mathbb{E}[\|\mathcal{D}_r^jX_n(s)\|^p] ds,
\end{align*}
where $C=C(T,p,L) >0$ is a constant.
From Gronwall's inequality, it follows that 
\begin{equation*}\label{e1.3.7}
	\mathbb{E}[\| \mathcal{D}_r^jX_{n}(t)\|^p]
	\leq C(1+\mathbb{E}[\|\xi\|^p])e^{Ct}
	\leq C(1+\mathbb{E}[\|\xi\|^p]).
\end{equation*}
According to Lemma \ref{L1.2.1-1}, we obtain 
\begin{equation*}\label{e1.3.9}
	\mathbb{E}[\|\mathcal{D}X_n(t)\|^p_{\mathbb{H}\otimes \mathbb{R}^d}] \leq C(1+\mathbb{E}[\|\xi\|^p]).
\end{equation*}	

By Remark \ref{L1.3.4} and \eqref{e1.3.4}, we have $X(t) \in \mathbb{D}^{1,p}(\mathbb{R}^d)$ and $\mathcal{D}X_n(t)$ converges to $\mathcal{D}X(t)$ in the weak topology of $L^p(\Omega; \mathbb{H} \otimes \mathbb{R}^d)$ as $n \to \infty$. Moreover,  $\mathcal{D}_r^jX(t) = \mathcal{D}_r^jX(t)\mathbbm{1}_{\{r \leq t\}}$ satisfies  for $0 \leq r \leq t \leq T$, 
\begin{equation}\label{e1.3.10}
	\begin{aligned}
		\mathcal{D}_r^jX(t) = &\ \int_{r}^{t} \bar{b}(s, \mathcal{L}[X(s)])\mathcal{D}_r^jX(s) ds \\
		&\ +\sigma^j(r, X(r), \mathcal{L}[X(r)])
		+\sum_{k=1}^{m}\int_{r}^{t} \bar{\sigma}^{k}(s, \mathcal{L}[X(s)])\mathcal{D}_r^jX(s) dW_s^k.
	\end{aligned}
\end{equation}	
Apply the same argument as  estimating $\mathcal{D}X_n(t)$ to above equation, we obtain
\begin{equation}\label{e1.3.10-1}
	\mathbb{E}[\|\mathcal{D}_r^jX(t)\|^p]  \vee \mathbb{E}[\|\mathcal{D}X(t)\|^p_{\mathbb{H}\otimes \mathbb{R}^d}] < C(1+\mathbb{E}[\|\xi\|^p]).
\end{equation}

Now we are ready to prove the existence of density function to solution $X(t)$.

\begin{theorem}\label{L1.3.6}
	Assume (H1)-(H3). Then for any $0< t \leq T$ the law of the solution $X(t)$ to \eqref{e1.3.2} is absolutely continuous with respect to the Lebesgue measure on $\mathbb{R}^d$, i.e.\  the density function of $X(t)$ exists.
\end{theorem}

\begin{proof}
	By virtue of Lemma \ref{L1.2.1}, we only need to investigate the invertibility of the Malliavin covariance matrix $Q(t)$ of $X(t)$. 
	Let $Y(t)$ and $Z(t)$ be solutions of the following two equations, respectively,
	\begin{eqnarray}
		\label{e1.3.11}
		Y(t) &= & I +  \int_{0}^{t} \bar{b}(s, \mathcal{L}[X(s)])Y(s) ds 
		+\sum_{k=1}^{m}\int_{0}^{t}  \bar{\sigma}^k(s, \mathcal{L}[X(s)]) Y(s) dW_s^k,
		\\
		\nonumber
		Z(t) &= & I  - \int_{0}^{t} Z(s)\bar{b}(s, \mathcal{L}[X(s)])  ds + \sum_{k=1}^{m}\int_{0}^{t} Z(s)\bar{\sigma}^k(s, \mathcal{L}[X(s)])\bar{\sigma}^k(s, \mathcal{L}[X(s)]) ds
		\\
		\nonumber 
		& & - \sum_{k=1}^{m} \int_{0}^{t}Z(s) \bar{\sigma}^{k}(s, \mathcal{L}[X(s)]) dW_s^k.
	\end{eqnarray}
	By It\^o formula, for any $t \in [0,T]$, there holds 
	\begin{equation*}
		\begin{aligned}
			Z(t)Y(t) = &\ I - \int_{0}^{t} Z(s)\bar{b}(s,\mathcal{L}[X(s)])Y(s) ds 
			\\ 
			&+\sum_{k=1}^{m} \int_{0}^{t} Z(s)\bar{\sigma}^k(s, \mathcal{L}[X(s)])\bar{\sigma}^k(s, \mathcal{L}[X(s)])Y(s) ds 
			\\
			&-\sum_{k=1}^{m} \int_{0}^{t} Z(s)\bar{\sigma}^k(s, \mathcal{L}[X(s)])Y(s) dW_s^k + \int_{0}^{t} Z(s)\bar{b}(s,\mathcal{L}[X(s)])Y(s) ds 
			\\
			&+ \sum_{k=1}^{m} \int_{0}^{t}Z(s)\bar{\sigma}^k(s, \mathcal{L}[X(s)])Y(s) dW_s^k 
			\\
			&- \sum_{k=1}^{m} \int_{0}^{t} Z(s)\bar{\sigma}^k(s, \mathcal{L}[X(s)])\bar{\sigma}^k(s, \mathcal{L}[X(s)])Y(s) ds\\
			= &\ I.
		\end{aligned}
	\end{equation*}
	Then, $Y(t)$ is invertible for all $t\in [0,T]$ and has continuous sample trajectories a.s. 
	Thus, there exists an event $\Omega_0\subset \Omega$ with $\mathbb{P}(\Omega_0) =1$, such that  for any  $\omega \in \Omega_0$, $\|Y^{\omega}(t)\|$ is continuous in $t\in [0,T]$, hence there exist a constant $\gamma=\gamma(\omega)>0$ such that $\|Y(t)\|\leq\gamma$. 
	By direct computation we can verify that the stochastic process $Y(t)Y^{-1}(r) \sigma^j(r,X(r),\mathcal{L}[X(r)])$ satisfies  \eqref{e1.3.10}, which implies
	\begin{equation*}
		\mathcal{D}_r^jX(t) =  Y(t)Y^{-1}(r) \sigma^j(r,X(r),\mathcal{L}[X(r)]) = Y(t)Y^{-1}(r) \sigma^j(r), \quad \forall 0 \leq r \leq t \leq T.
	\end{equation*}
	Denote by $Y(t) = (Y_1(t), \dots, Y_d(t))^T$, then the Malliavin covariance matrix $Q(t)$ satisfies
	\begin{equation*}
		\begin{aligned}
			Q_{i,l}(t) = &\ ( \mathcal{D}X^i(t), \mathcal{D}X^l(t))_{\mathbb{H}} = \sum_{j=1}^{m}\int_{0}^{T}\mathcal{D}_r^j X^i(t)\mathcal{D}_r^j X^l(t) dr = \sum_{j=1}^{m}\int_{0}^{t}\mathcal{D}_r^j X^i(t)\mathcal{D}_r^j X^l(t) dr
			\\
			= &\ \sum_{j=1}^{m} \int_{0}^{t} Y_{i}(t)Y^{-1}(r)\sigma^j(r) Y_{l}(t)Y^{-1}(r)\sigma^j(r) dr
			\\
			= &\  Y_{i}(t) \big(\int_{0}^{t} Y^{-1}(r)\big(\sum_{j=1}^{m} \sigma^j(r) (\sigma^j)^{T}(r)\big)(Y^{-1})^{T}(r) dr\big) Y^{T}_{l}(t)
			\\
			= &\ Y_i(t) (\int_{0}^{t} Y^{-1}(r) A(r)(Y^{-1})^T(r) dr) Y_l^T(t)
			\\ 
			:= &\  Y_{i}(t) P(t) Y_{l}^{T}(t).
		\end{aligned}
	\end{equation*}
	
	For any $ u \in \mathbb{R}^d$, let $v(r) = (Y^{-1})^T(r) u$ and then we have
	\begin{equation*}
		\|u\| = \|Y^T(r)v(r)\| \leq \|Y^T(r)\| \|v(r)\| \leq \gamma \|v(r)\|.
	\end{equation*}
	Thus 
	\begin{equation*}
		\|v(r)\| \geq \frac{\|u\|}{\gamma}.
	\end{equation*}
	Then we have
	\begin{equation*}\label{e1.3.12}
		\begin{aligned}
			u^T P(t) u &= \int_{0}^{t} u^T Y^{-1}(r)A(r)(Y^{-1})^T(r) u dr
			= \int_{0}^{t} v^T(r)A(r)v(r) dr \\
			&\geq \int_{0}^{t} \lambda\| v(r)\|^2 dr
			\geq \frac{t\lambda}{\gamma^2}\|u\|^2,
		\end{aligned}
	\end{equation*}
	which  implies  $Q(t)$ is invertible  for all $\omega \in \Omega_0$ and $t\in (0,T]$. According to Lemma \ref{L1.2.1}, the density function of $X(t)$ exists for all $t \in (0,T]$.
\end{proof}

%%%%%%%%%%%%%%%%%%%%%%%%%%%%%%%%%%%%%%%%%%%%%%%%%%%%%%%%%%%%%%%%%%%%%%%%%%%%%%%%%%%%%%%%%%%%%%%%%%%%%%%%%%%%%%%%%%%%%%%%%%
\section{Smoothness of density function}
In this section, we will investigate the finite order regularity of the density function $p(t,x)$ of the solution $X(t)$ in terms of certain smoothness of the coefficients $b$ and $\sigma$. For simplifying notations, we use $D^nb(t)$ and $D^n\sigma(t)$ $(n\geq 0)$ to denote the $n$-th order Fr\'echet derivatives of $b(t,x,\mu)$ and $\sigma(t,x,\mu)$ with respect to $x$, respectively. Here we stipulate $D^0b(t) = b(t)$ and $D^0\sigma(t) = \sigma(t)$. Instead of (H1) and (H2), we assume
\begin{itemize}
	\item[(H1$^{\prime}$)]$\xi \in L_d^{\infty -}(\Omega)$ is independent of $W_t$ for $t \in (0,T]$.
	\item[(H2$^{\prime}$)]
	There exists an integer $N \geq 0$ such that $D^nb(t,x,\mu)$ and $ D^n\sigma(t,x,\mu)$ are continuous for  $0 \leq n \leq N+2$ and bounded for  $1 \leq n \leq N+2$. Moreover, there exists a constant $L > 0$ such that  for all $n\geq 0$, $t \in  [0,T]$, $x, x' \in \mathbb{R}^d$ and $\mu, \mu' \in \mathcal{P}_2(\mathbb{R}^d)$ there hold 
	\begin{equation*}
		\begin{aligned}
			\| D^{n}b(t,x,\mu) - D^{n}b(t, x', \mu')\|_{({\mathbb{R}^d})^{\otimes (n+1)}} \leq L(\|x-x'\| + W_2(\mu,\mu')) ,\\
			\| D^{n}\sigma(t,x,\mu) - D^{n} \sigma(t, x', \mu') \|_{({\mathbb{R}^d})^{\otimes (n+1)}\otimes \mathbb{R}^m} \leq L(\|x- x'\| + W_2(\mu,\mu')).
		\end{aligned}
	\end{equation*}
\end{itemize}

We start with investigating the regularity of the  Malliavin covariance matrix $Q$.

\begin{lemma}\label{L1.4.1}
	Assume (H1$^{\prime}$), (H2$^{\prime}$) and (H3). Then for any $0 < t \leq T$ the  Malliavin covariance matrix $Q(t)$ of the solution $X(t)$ to \eqref{e1.3.1} satisfies $(detQ(t))^{-1} \in L^{\infty -}(\Omega)$.
\end{lemma}

\begin{proof}
	Obviously, (H1$^{\prime}$) and (H2$^{\prime}$) imply (H1) and (H2), respectively. By the analysis in previous section, we have $X(t) \in \mathbb{D}^{1,\infty} (\mathbb{R}^d)$.  Since the functions $b$ and $\sigma$ are continuously differentiable with respect to $x$, we have $\bar{b} = Db$ and $\bar{\sigma}^k = D\sigma^k$ in \eqref{e1.3.11}. By $Y(t)$ we denote the solution to \eqref{e1.3.11}. Then,  $Y(t)$ is invertible for all $t\in [0,T]$ and has continuous sample trajectories a.s. Thus, there exists an event $\Omega_0\subset \Omega$ with $\mathbb{P}(\Omega_0) =1$, such that  for any  $\omega \in \Omega_0$, $\|Y^{\omega}(t)\|$ and $\|(Y^{\omega})^{-1}(t)\|$ are continuous on $[0,T]$, hence  there exist a constant $\gamma=\gamma(\omega) > 0$ such that 
	\begin{equation}\label{e1.4.6}
		\|Y(t)\| \vee \|Y^{-1}(t)\| \leq\gamma.
	\end{equation}
	Due to \eqref{e1.3.10-1}, for any $p \geq 1$ and  $1\leq i,l \leq d$, there holds 
	\begin{equation}\label{e1.4.6-1}
		\begin{aligned}
			\mathbb{E}[|Q_{i,l}(t)|^p] =&\ \mathbb{E}[|(\mathcal{D}X^i(t), \mathcal{D}X^l(t))_{\mathbb{H}}|^p] 
			\leq \mathbb{E}[\|\mathcal{D}X^i(t)\|^p_{\mathbb{H}} \|\mathcal{D}X^l(t)\|_{\mathbb{H}}^p]\\
			\leq&\ \frac12\mathbb{E}[\|\mathcal{D}X^i(t)\|_{\mathbb{H}}^{2p}] +\frac12\mathbb{E}[\|\mathcal{D}X^l(t)\|_{\mathbb{H}}^{2p}]\leq  C(1+\mathbb{E}[\|\xi\|^{2p}]),
		\end{aligned}
	\end{equation} 
	where $C = C(T,p) > 0$ is a constant.
	Thus $Q_{i,l}$ has finite moments of all orders. Similar to the analysis in proving Theorem \ref{L1.3.6}, we have $Q(t)$ is invertible a.s.\ and satisfies
	\begin{equation*}
		Q(t) = Y(t)(\int_{0}^{t} Y^{-1}(r) A(r)(Y^{-1}(r))^T dr) Y^T(t).
	\end{equation*}  
	For any $u \in \mathbb{R}^d$ with $\| u \| = 1$ and $0 < r \leq t \leq T$, define $v(r, t) = (Y^{-1}(r))^T Y^T(t) u$. From \eqref{e1.4.6}, it follows that 
	\begin{equation*}
		1 = \|u\| = \| (Y^{-1}(t))^T Y^T(r) v(t ,r)\| \leq \gamma^2 \|v(r,t)\|,
	\end{equation*}
	which implies	$\|v(r,t) \| \geq \frac{1}{\gamma^2}$.
	Then, there holds
	\begin{equation*}
		\begin{aligned}
			u^TQ(t)u &= u^TY(t)\big(\int_{0}^{t} Y^{-1}(r)A(r)(Y^{-1}(r))^T dr\big) Y^T(t)u
			= \int_{0}^{t} v^T(r,t) A(r) v(r,t) dr \\
			&\geq \int_{0}^{t} \lambda\| v(r, t)\|^2 dr
			\geq \frac{t\lambda}{\gamma^4}.
		\end{aligned}
	\end{equation*}
	Taking $\epsilon_0 = \frac{t\lambda}{2\gamma^4}$, then for any $0 < \epsilon \leq \epsilon_0$ we obtain that
	\begin{equation}\label{e1.4.7}
		\mathbb{P}\{ u^TQ(t)u \leq \epsilon \} = 0 < \epsilon^p.
	\end{equation} 
	Noticing the estimates  \eqref{e1.4.6-1} and \eqref{e1.4.7},	from  Lemma \ref{L1.2.3}, it follows that  $(detQ(t))^{-1} \in L^p(\Omega)$ for all $p \geq 2$. Thus the proof is completed.
\end{proof}  

The next theorem concerns with the first order regularity of the density function $p(t,x)$ with respect to $x$.

\begin{theorem}\label{L1.4.2}
	Assume (H1$^{\prime}$), (H2$^{\prime}$) and (H3) hold for $N=1$. Let $p(t,x)$  be the density function of the solution $X(t)$ to \eqref{e1.3.1}. Then $p(t,x)$ together with its first order Fr\'echet derivative with respect to $x$ are continuous and bounded.
\end{theorem}

\begin{proof}
	According to Lemmas \ref{L1.2.2} and \ref{L1.4.1}, we only need to prove $X(t) \in \mathbb{D}^{3, \infty}(\mathbb{R}^d)$. From the estimate \eqref{e1.3.10-1}, it follows that $X(t) \in \mathbb{D}^{1, \infty}(\mathbb{R}^d)$. 
	
	Now, we estimate the second order  Malliavin derivative of $X(t)$. By virtue of \eqref{e1.3.10} and the smoothness of $b$ and $\sigma$,  for any integer $1 \leq j_1 \leq m$, the $j_1$-th component $\mathcal{D}_{r_1}^{j_1}X(t) = \mathcal{D}_{r_1}^{j_1}X(t)\mathbbm{1}_{\{r_1 \leq t\}}$ $(r_1 \in [0,T])$  of the Malliavin derivative $\mathcal{D}X(t)$ satisfies an SDE for $0 \leq r_1 \leq t \leq T$ 
	\begin{equation*}\label{e1.4.1}
		\mathcal{D}_{r_1}^{j_1}X(t) =  \int_{r_1}^{t} Db(s)\mathcal{D}_{r_1}^{j_1}X(s) ds  + \sigma^{j_1}(r_1) + \sum_{k=1}^{m}\int_{r_1}^{t} D\sigma^k(s)\mathcal{D}_{r_1}^{j_1}X(s) dW_s^k.
	\end{equation*}
	Taking Malliavin derivative on both sides of above equation,  we have that for any $1 \leq j_1, j_2 \leq m$ the $(j_1, j_2)$-th component $\mathcal{D}_{r_1,r_2}^{j_1,j_2}X(t) = \mathcal{D}_{r_1,r_2}^{j_1,j_2}X(t)\mathbbm{1}_{\{r_1,r_2 \leq t\}}$ $(r_1, r_2 \in [0,T])$ of the second order Malliavin derivative $\mathcal{D}^2X(t)$ satisfies an SDE for all  $ 0 \leq r_1, r_2 \leq t \leq T$  
	\begin{equation}\label{e1.4.2}
		\begin{aligned}
			&\mathcal{D}_{r_1,r_2}^{j_1,j_2}X(t)= \ \int_{r_1 \vee r_2}^{t}\big(D^2b(s)(\mathcal{D}_{r_1}^{j_1}X(s),\mathcal{D}_{r_2}^{j_2}X(s))
			+Db(s)\mathcal{D}_{r_1,r_2}^{j_1,j_2}X(s)\big) ds \\
			&\ \quad +D\sigma^{j_1}(r_1)\mathcal{D}_{r_2}^{j_2}X(r_1) + 
			D\sigma^{j_2}(r_2)\mathcal{D}_{r_1}^{j_1}X(r_2)\\
			&\ \quad + \sum_{k=1}^{m} \int_{r_1 \vee r_2}^{t} \big(D^2\sigma^k(s)(\mathcal{D}_{r_1}^{j_1}X(s),\mathcal{D}_{r_2}^{j_2}X(s))
			+D\sigma^k(s)\mathcal{D}_{r_1,r_2}^{j_1,j_2}X(s) \big) dW_s^k.
		\end{aligned}
	\end{equation}	
	In order to simplify the proof, we introduce some abbreviations.
	For any   $1\leq k, j_1, j_2 \leq m$ and $t, r_1, r_2 \in [0,T]$, let
	\begin{equation*}\label{e1.4.2-1}
		\begin{aligned}
			\Gamma^k_{j_1}(t, r_1)\! &\ := \mathcal{D}^{j_1}_{r_1}\sigma^k(t) = D\sigma^k(t)\mathcal{D}_{r_1}^{j_1}X(t),
			\\
			B_{j_1,j_2}(t, r_1, r_2)\! &\ :=  \mathcal{D}_{r_1, r_2}^{j_1, j_2} b(t)=  D^2b(t)(\mathcal{D}_{r_1}^{j_1}X(t)  \mathcal{D}_{r_2}^{j_2}X(t) ) + Db(t)\mathcal{D}_{r_1,r_2}^{j_1,j_2}X(t),
			\\
			\Gamma^{k}_{j_1,j_2}(t, r_1, r_2)\! &\ := \mathcal{D}_{r_1, r_2}^{j_1, j_2} \sigma^{k}(t)=  D^2\sigma^{k}(t)(\mathcal{D}_{r_1}^{j_1}X(t) , \mathcal{D}_{r_2}^{j_2}X(t) ) +
			D\sigma^{k}(t)\mathcal{D}_{r_1,r_2}^{j_1,j_2}X(t). 
		\end{aligned}
	\end{equation*}	
	Then,  \eqref{e1.4.2} can be rewritten as  
	\begin{equation*}\label{e1.4.2-2}
		\begin{aligned}
			\mathcal{D}_{r_1,r_2}^{j_1,j_2}X(t)= &\ \int_{r_1 \vee r_2}^{t} B_{j_1,j_2}(s,r_1,r_2) ds  +\Gamma^{j_1}_{j_2}(r_1, r_2) +  \Gamma^{j_2}_{j_1}(r_2, r_1)\\
			&\ + \sum_{k=1}^{m} \int_{r_1 \vee r_2}^{t} \Gamma^{k}_{j_1,j_2}(s, r_1, r_2) dW_s^k.
		\end{aligned}
	\end{equation*}
	Applying H\"older inequality, \eqref{e1.3.10-1} and BDG inequality, we have
	\begin{equation}\label{e1.4.2-3}
		\begin{aligned}
			&\ \mathbb{E}[\|\mathcal{D}_{r_1,r_2}^{j_1,j_2}X(t)\|^p] \\
			\leq&\  (m+3)^{p-1}\mathbb{E}[\|\!\!\int_{r_1 \vee r_2}^{t} \!\!\!\!B_{j_1,j_2}(s,r_1,r_2) ds\|^p + \|\Gamma^{j_1}_{j_2}(r_1, r_2) \|^p  + \|\Gamma^{j_2}_{j_1}(r_2, r_1)\|^p]\\ 
			&\ + (m+3)^{p-1}\sum_{k=1}^{m} \mathbb{E}[\sup_{r_1 \vee r_2 \leq \tau \leq t}\|\int_{r_1 \vee r_2}^{\tau} \Gamma^{k}_{j_1,j_2}(s, r_1, r_2) dW_{s}^k \|^p] \\
			\leq &\ CT^{p-1}\mathbb{E}[\int_{r_1 \vee r_2}^{t} \|B_{j_1,j_2}(s,r_1,r_2)\|^p ds] + 2C(1+\mathbb{E}[\|\xi\|^p])\\
			&\ + C\sum_{k=1}^{m} C(p)\mathbb{E}[(\int_{r_1 \vee r_2}^{t} \|\Gamma^{k}_{j_1,j_2}(s, r_1, r_2)\|^2 ds)^{\frac{p}{2}}]\\
			\leq &\  C\int_{r_1 \vee r_2}^{t} \mathbb{E}[\|B_{j_1,j_2}(s,r_1,r_2)\|^p] ds + C(1+\mathbb{E}[\|\xi\|^p])\\
			&\ + C\sum_{k=1}^{m}T^{\frac{p-2}{2}}\int_{r_1 \vee r_2}^{t} \mathbb{E}[\|\Gamma^{k}_{j_1,j_2}(s, r_1, r_2)\|^p] ds.
		\end{aligned}
	\end{equation}
	
	Notice that
	\begin{equation}\label{e1.4.2-4}
		\begin{aligned}
			&\ \mathbb{E}[\|B_{j_1,j_2}(s,r_1,r_2)\|^p] \\
			\leq &\ 2^{p-1}\mathbb{E}[\|D^2b(s)(\mathcal{D}_{r_1}^{j_1}X(s), \mathcal{D}_{r_2}^{j_2}X(s))\|^p + \|Db(s)\mathcal{D}_{r_1,r_2}^{j_1,j_2}X(s)\|^p] \\
			\leq &\ CM^p\mathbb{E}[\|\mathcal{D}_{r_1}^{j_1}X(s)\|^p \|\mathcal{D}_{r_2}^{j_2}X(s)\|^p] + CM^p\mathbb{E}[\|\mathcal{D}_{r_1,r_2}^{j_1,j_2}X(s)\|^p] \\
			\leq &\ \frac{1}{2}C\mathbb{E}[\|\mathcal{D}_{r_1}^{j_1}X(s)\|^{2p}] + \frac{1}{2}C\mathbb{E}[\|\mathcal{D}_{r_2}^{j_2}X(s)\|^{2p}] + C\mathbb{E}[\|\mathcal{D}_{r_1,r_2}^{j_1,j_2}X(s)\|^p]\\
			\leq &\ C(1+\mathbb{E}[\|\xi\|^{2p}]) + C\mathbb{E}[\|\mathcal{D}_{r_1,r_2}^{j_1,j_2}X(s)\|^p].
		\end{aligned}
	\end{equation}
	In a similar way, we obtain
	\begin{equation}\label{e1.4.2-5}
		\mathbb{E}[\|\Gamma^k_{j_1,j_2}(s,r_1,r_2)\|^p] \leq C(1+\mathbb{E}[\|\xi\|^{2p}]) + C\mathbb{E}[\|\mathcal{D}_{r_1,r_2}^{j_1,j_2}X(s)\|^p].
	\end{equation}
	Substituting  \eqref{e1.4.2-4} and \eqref{e1.4.2-5} into \eqref{e1.4.2-3}, and then applying Gronwall's inequality yields
	\begin{equation*}\label{e1.4.2-6}
		\mathbb{E}[\|\mathcal{D}_{r_1,r_2}^{j_1,j_2}X(t)\|]^p \leq C(1+\mathbb{E}[\|\xi\|^{p}]+ \mathbb{E}[\|\xi\|^{2p}]),
	\end{equation*} 
	where $C = C(T,p)$ is a constant.
	Therefore, from Lemma \ref{L1.2.1-1}, it follows that
	\begin{equation*}
		\mathbb{E}[\|\mathcal{D}^2X(t)\|^p_{\mathbb{H}^{\otimes 2}\otimes \mathbb{R}^d}] \leq C,\quad \forall t \in [0,T],
	\end{equation*} 
	which implies $X(t) \in \mathbb{D}^{2, \infty}(\mathbb{R}^d)$. 
	
	Finally, we estimate the third order Malliavin derivative $\mathcal{D}^3X(t)$. In a similar way, we define some notations  for  $1\leq k, j_1, j_2, j_3 \leq m$ and $t, r_1, r_2, r_3 \in [0,T]$ 
	\begin{equation*}
		\begin{aligned}
			&\ B_{j_1,j_2,j_3} (t, r_1, r_2, r_3) := \mathcal{D}_{r_1, r_2, r_3}^{j_1, j_2, j_3}b(t) \\
			&\ \qquad =  D^3b(t)(\mathcal{D}_{r_1}^{j_1}X(t),\mathcal{D}_{r_2}^{j_2}X(t), \mathcal{D}_{r_3}^{j_3}X(t)) \\
			&\ \qquad +D^2b(t)(\mathcal{D}_{r_1, r_3}^{j_1, j_3}X(t),\mathcal{D}_{r_2}^{j_2}X(t)) 
			+D^2b(t)(\mathcal{D}_{r_1}^{j_1}X(t),\mathcal{D}_{r_2, r_3}^{j_2, j_3}X(t))\\
			&\ \qquad +D^2b(t)(\mathcal{D}_{r_1, r_2}^{j_1, j_2}X(t),\mathcal{D}_{r_3}^{j_3}X(t))
			+Db(t)\mathcal{D}_{r_1, r_2, r_3}^{j_1, j_2, j_3}X(t)\\
			&\ \Gamma^{k}_{j_1,j_2,j_3} (t, r_1, r_2, r_3) :=  \mathcal{D}_{r_1, r_2, r_3}^{j_1, j_2, j_3}\sigma^k(t)\\
			&\	\qquad = D^3\sigma^k(t)(\mathcal{D}_{r_1}^{j_1}X(t),\mathcal{D}_{r_2}^{j_2}X(t), \mathcal{D}_{r_3}^{j_3}X(t))\\
			&\ \qquad +D^2\sigma^k(t)(\mathcal{D}_{r_1, r_3}^{j_1, j_3}X(t),\mathcal{D}_{r_2}^{j_2}X(t))
			+D^2\sigma^k(t)(\mathcal{D}_{r_1}^{j_1}X(t),\mathcal{D}_{r_2, r_3}^{j_2, j_3}X(t))\\
			&\ \qquad +D^2\sigma^k(t)(\mathcal{D}_{r_1, r_2}^{j_1, j_2}X(t),\mathcal{D}_{r_3}^{j_3}X(t))
			+D\sigma^k(t)\mathcal{D}_{r_1, r_2, r_3}^{j_1, j_2, j_3}X(t). 	
		\end{aligned}
	\end{equation*}	
	
	Noticing that $\mathcal{D}_{r_1, r_2, r_3}^{j_1,j_2,j_3}X(t) = \mathcal{D}_{r_1, r_2, r_3}^{j_1,j_2,j_3}X(t)\mathbbm{1}_{\{r_1,r_2,r_3 \leq t\}}$, and then taking Malliavin derivative on both sides of \eqref{e1.4.2} yields  for $0 \leq r_1, r_2, r_3 \leq t \leq T$ 
	\begin{equation*}
		\begin{aligned}
			\mathcal{D}_{r_1, r_2, r_3}^{j_1,j_2,j_3}X(t)
			=&\ \int_{r_1 \vee r_2 \vee r_3}^{t} B_{j_1,j_2,j_3} (s, r_1, r_2, r_3) ds + 
			\Gamma^{j_1}_{j_2,j_3}(r_1, r_2, r_3) + \Gamma^{j_2}_{j_1,j_3}(r_2, r_1, r_3) \\
			&\  + \Gamma^{j_3}_{j_1,j_2}(r_3, r_1, r_2) + \sum_{k=1}^{m} \int_{r_1 \vee r_2 \vee r_3}^{t} \Gamma^{k}_{j_1,j_2,j_3} (s, r_1, r_2, r_3) dW_s^k.
		\end{aligned}
	\end{equation*}
	Repeating  the procedure for estimating $\mathcal{D}^2X(t)$,  we can prove  $X(t) \in \mathbb{D}^{3,\infty}(\mathbb{R}^d)$ for all $t \in [0,T]$. Thus, the proof is completed. 
\end{proof}

In order to study higher order regularity of density function $p(t,x)$ with respect to $x$, we need to estimate higher order Malliavin derivative of $X(t)$. Before this, we introduce some abbreviations to make statement concise. For any integers $n \geq 1$ and $1\leq \alpha \leq n$, define a subset $K = \{\eta_1 < \cdots < \eta_{\alpha}\} \subset \{1, \dots, n\}$.  For any integer $1\leq \theta \leq n$, let  $K_1$, $\cdots$, $K_{\theta}$ be a family of disjoint such sets so that they form a partition of $\{1, \dots, n\}$. Denote by  $\Lambda$  the set of all such partitions of $\{1, \dots, n\}$. Let $j(K) = (j_{\eta_1}, \dots, j_{\eta_{\alpha}})$ and $r(K) = (r_{\eta_1}, \dots, r_{\eta_{\alpha}})$ be two multi-indices with $j_{\eta_{i}} \in \{1,\dots,m\}$ and $r_{\eta_i} \in [0,T]$. Define
\begin{equation*}%\label{e1.4.8}
	\begin{aligned}
		B_{j_1,\dots,j_n}(t, r_1,\dots, r_n) := &\ \mathcal{D}_{r_1,\dots, r_n}^{j_1,\dots,j_n}b(t) = \sum_{\Lambda} D^{\theta}b(t)(\mathcal{D}_{r(K_1)}^{j(K_1)}X(t), \cdots, \mathcal{D}_{r(K_{\theta})}^{j(K_{\theta})}X(t)) \\
		= &\ R_{j_1,\dots,j_{n}}(t, r_1,\dots, r_{n}) + Db(t)\mathcal{D}_{r_1, \dots, r_{n}}^{j_1, \dots, j_{n}}X(t), 
	\end{aligned}
\end{equation*}
where 
\begin{equation*}
	R_{j_1,\dots,j_{n}}(t, r_1,\dots, r_{n}) = \sum_{\Lambda, \theta \not= 1} D^{\theta}b(t)(\mathcal{D}_{r(K_1)}^{j(K_1)}X(s), \cdots, \mathcal{D}_{r(K_{\theta})}^{j(K_{\theta})}X(t)). 
\end{equation*}
Similarly, for  $k = 1,\dots,m$ define 
\begin{equation*}%\label{e1.4.9}
	\begin{aligned}
		\Gamma_{j_1,\dots,j_n}^k(t, r_1,\dots, r_n) := &\ \mathcal{D}_{r_1,\dots, r_n}^{j_1,\dots,j_n}\sigma^k(t)
		= P_{j_1,\dots,j_{n}}^k(t, r_1,\dots, r_{n}) + D\sigma^k(t)\mathcal{D}_{r_1, \dots, r_{n}}^{j_1, \dots, j_{n}}X(t),
	\end{aligned}
\end{equation*} 
where 
\begin{equation*}
	P_{j_1,\dots,j_{n}}^k(t, r_1,\dots, r_{n}) = \sum_{\Lambda, \theta \not= 1} D^{\theta}\sigma^k(t)(\mathcal{D}_{r(K_1)}^{j(K_1)}X(t), \cdots, \mathcal{D}_{r(K_{\theta})}^{j(K_{\theta})}X(t)).
\end{equation*} 

Now we are ready to investigate the quantitative dependence on the smoothness of coefficients for the regularity of the density function $p(t,x)$.

\begin{theorem}\label{L1.4.3}
	Assume (H1$^{\prime}$), (H2$^{\prime}$) and (H3). Let $p(t,x)$  be the density function of the solution $X(t)$ to \eqref{e1.3.1}. Then $p(t,x)$ together with its Fr\'echet derivatives with respect to $x$ up to order $N$ is continuous and bounded.
\end{theorem}

\begin{proof} 
	First, we will apply mathematical induction to prove that for any $1 \leq n \leq N+2$ and $1 \leq j_1,\dots,j_n \leq m$, the $(j_1, \dots, j_n)$-th component $\mathcal{D}_{r_1, \dots, r_n }^{j_1, \dots, j_n}X(t) = \mathcal{D}_{r_1, \dots, r_n }^{j_1, \dots, j_n}X(t)\mathbbm{1}_{\{r_1,\dots, r_n \leq t\}}$ $(r_1, \dots, r_n \in [0,T])$ of the $n$-th order Malliavin derivative $\mathcal{D}^nX(t)$ satisfies an SDE for $0 \leq r_1,\dots, r_n \leq t \leq T$ 
	\begin{equation}\label{e1.4.10}
		\begin{aligned}
			\mathcal{D}_{r_1, \dots, r_n }^{j_1, \dots, j_n}X(t) = 
			&\ \int_{r_1 \vee \cdots \vee r_n}^{t} B_{j_1,\dots,j_n}(s, r_1, \dots, r_n) ds \\
			&\ +\sum_{l = 1}^{n} \Gamma^{j_{l}}_{j_1,\dots, j_{l-1}, j_{l+1}, \dots, j_n}(r_{l}, r_1,\dots,r_{l -1}, r_{l+1}, \dots, r_n)\\
			&\ +\sum_{k=1}^{m} \int_{r_1 \vee \cdots \vee r_n}^{t} \Gamma_{j_1,\dots,j_n}^{k}(s, r_1, \dots, r_n) dW_s^k,
		\end{aligned}
	\end{equation}     
	and for any $p \geq 2$ 
	\begin{equation}\label{e1.4.10-1}
		\begin{aligned}
			\mathbb{E}[\|\mathcal{D}_{r_1, \dots, r_n }^{j_1, \dots, j_n}X(t) \|^p] \leq C\quad \text{and } X(t) \in \mathbb{D}^{n, p}(\mathbb{R}^d),
		\end{aligned}
	\end{equation}
	where $C>0$ is a constant.
	
	In fact, Theorem \ref{L1.4.2} implies that \eqref{e1.4.10} and \eqref{e1.4.10-1} hold for $n=1,2,3$. We suppose that \eqref{e1.4.10} and \eqref{e1.4.10-1} are valid  for $ 3 \leq n < N+2$, and  we will show these  hold true for $n+1$. 
	
	Taking Malliavin derivative $\mathcal{D}_r^j$ on $B_{j_1, \dots, j_n}(t,r_1,\dots,r_n)$ and $\Gamma_{j_1, \dots, j_n}^k(t,r_1,\dots,r_n)$, respectively, yields
	\begin{equation*}
		\begin{aligned}
			\mathcal{D}_r^j B_{j_1, \dots, j_n}(t,r_1,\dots,r_n) &\ = B_{j_1,\dots,j_n,j}(t, r_1,\dots,r_n, r),
			\\
			\mathcal{D}_r^j \Gamma_{j_1, \dots, j_n}^k(t,r_1,\dots,r_n) &\ = \Gamma_{j_1,\dots,j_n,j}^k(t, r_1,\dots,r_n, r).
		\end{aligned}
	\end{equation*}  
	For any $1\leq j_{n+1} \leq m$, take Malliavin derivative $\mathcal{D}_{r_{n+1}}^{j_{n+1}}$ on both sides of \eqref{e1.4.10}, then we have for $0 \leq r_1, \dots, r_{n+1} \leq t \leq T$ 
	\begin{equation*}
		\begin{aligned}
			&\ \mathcal{D}_{r_1, \dots, r_{n+1} }^{j_1, \dots, j_{n+1}}X(t) = \mathcal{D}_{r_{n+1}}^{j_{n+1}}\mathcal{D}_{r_1, \dots, r_n }^{j_1, \dots, j_n}X(t) 
			\\
			= &\ \int_{ r_1 \vee \cdots \vee r_{n+1}}^{t} B_{j_1,\dots, j_{n+1}}(s, r_1, \dots, r_{n+1}) ds\\
			&\ +\sum_{l = 1}^{n} \Gamma^{j_{l}}_{j_1,\dots, j_{l-1}, j_{l+1}, \dots, j_n, j_{n+1}}(r_{l}, r_1,\dots,r_{l -1}, r_{l+1}, \dots, r_n, r_{n+1})\\
			&\ + \Gamma_{j_1,\dots,j_n}^{j_{n+1}}(r_{n+1}, r_1, \dots, r_n)+\sum_{k=1}^{m} \int_{r_1 \vee \cdots \vee r_{n+1}}^{t} \Gamma_{j_1,\dots, j_{n+1}}^{k}(s, r_1, \dots, r_{n+1}) dW_s^k
			\\
			=&\ \int_{r_1 \vee \cdots \vee r_{n+1}}^{t} B_{j_1,\dots,j_{n+1}}(s, r_1, \dots, r_{n+1}) ds \\ 
			&\ +\sum_{l = 1}^{n+1} \Gamma^{j_{l}}_{j_1,\dots, j_{l-1},j_{l+1}, \dots, j_{n+1}}(r_{l}, r_1,\dots,r_{l -1}, r_{l+1}, \dots, r_{n+1})\\
			&\ +\sum_{k=1}^{m} \int_{r_1 \vee \cdots \vee r_{n+1}}^{t} \Gamma_{j_1,\dots,j_{n+1}}^{k}(s, r_1, \dots, r_{n+1}) dW_s^k, 
		\end{aligned}
	\end{equation*}
	which means \eqref{e1.4.10} holds true for $n+1$. 
	
	Since $R_{j_1,\dots,j_{n+1}}(t, r_1,\dots, r_{n+1})$ and $\Gamma^{j_{l}}_{j_1,\dots, j_{l-1}, j_{l+1}, \dots, j_{n+1}}(r_{l}, r_1,\dots,r_{l -1}, r_{l+1}, \dots, r_{n+1})$ only depend on  the Fr\'echet derivatives  $D^{\theta}b(s)$ for $1\leq\theta\leq n+1$ and the Malliavin derivatives $\mathcal{D}^{\alpha}X(t)$ for $\alpha \leq n$, by (H2$^\prime$) and \eqref{e1.4.10-1}, we have 
	\begin{equation*}
		\begin{aligned}
			& \mathbb{E}[\|R_{j_1,\dots,j_{n+1}}(t, r_1,\dots, r_{n+1})\|^p] \leq C,\\
			& \mathbb{E}[\| \Gamma^{j_{l}}_{j_1,\dots, j_{l-1}, j_{l+1}, \dots, j_{n+1}}(r_{l}, r_1,\dots,r_{l -1}, r_{l+1}, \dots, r_{n+1})\|^p] \leq C
		\end{aligned}
	\end{equation*}
	for some constant $C >0$.  Then, we obtain
	\begin{equation*}
		\mathbb{E}[\|B_{j_1,\dots,j_{n+1}}(s, r_1,\dots, r_{n+1})\|^p]
		\leq   C(1+\mathbb{E}[\|\mathcal{D}_{r_1, \dots, r_{n+1}}^{j_1, \dots, j_{n+1}}X(s)\|^p]).
	\end{equation*}
	Similarly, we also have 
	\begin{equation*}
		\mathbb{E}[\|\Gamma^k_{j_1,\dots,j_{n+1}}(s, r_1,\dots, r_{n+1})\|^p]
		\leq  C(1+\mathbb{E}[\|\mathcal{D}_{r_1, \dots, r_{n+1}}^{j_1, \dots, j_{n+1}}X(s)\|^p]).
	\end{equation*} 
	Repeating the proof for \eqref{e1.4.2-3} leads to   
	\begin{align*}
		&\ \mathbb{E}[\|\mathcal{D}_{r_1, \dots, r_{n+1}}^{j_1, \dots, j_{n+1}}X(t)\|^p]
		\leq  C\int_{r_1 \vee \cdots \vee r_{n+1}}^{t} \mathbb{E}[\|B_{j_1,\dots,j_{n+1}}(s, r_1,\dots, r_{n+1})\|^p] ds\\
		&\ +C + C\sum_{k=1}^{m} \int_{r_1 \vee \cdots \vee r_{n+1}}^{t} \mathbb{E}[\| \Gamma_{j_1,\dots,j_{n+1}}^k(s, r_1,\dots, r_{n+1})\|^p] ds.
	\end{align*}
	By  Gronwall's inequality, we obtain 
	\begin{equation*}\label{e1.4.13}
		\mathbb{E}[\|\mathcal{D}_{r_1, \dots, r_{n+1}}^{j_1, \dots, j_{n+1}}X(t)\|^p] \leq  C,
	\end{equation*} 
	where $C > 0 $ is a constant.
	
	According to Lemma \ref{L1.2.1-1}, we have 	$\mathbb{E}[\|\mathcal{D}^{n+1}X(t)\|^p_{\mathbb{H}^{\otimes n+1}\otimes \mathbb{R}^d}] \leq C$, thus  $X(t) \in \mathbb{D}^{n+1, p}(\mathbb{R}^d)$ for all $t\in[0,T]$.   
	
	Especially, as $n = N+2$ we have $X(t) \in \mathbb{D}^{N+2, p}(\mathbb{R}^d)$. Due to the arbitrariness of choice for $p$, we obtain $X(t) \in \mathbb{D}^{N+2, \infty}(\mathbb{R}^d)$, which together with Lemma \ref{L1.4.1} implies that $p(t,\cdot) \in C_b^N(\mathbb{R}^d)$ due to Lemma \ref{L1.2.2}.
\end{proof}

%%%%%%%%%%%%%%%%%%%%%%%%%%%%%%%%%%%%%%%%%%%%%%%%%%%%%%%%%%%%%%%%%%%%%%%%%%%%%%%%%%%%%%%%%%%%%%%%%%%%%%%%%%%%%%%%%%%%%%%%%% 
\section{Numerical experiments} 
The difficulties on studying  MVE usually arise from the unknown distribution of its solution. Our study in this paper ensures the existence and proper smoothness of the density function of solutions to MVEs under certain regularity assumptions for coefficients. Thus, the density functions  satisfy corresponding Fokker-Plank  equations, which enable us to independently study their properties and seek  approximations.
In this section, we will present several numerical experiments to illustrate approximations of  density functions obtained by solving Fokker-Plank equations.

\begin{example}\label{E1.5.1}
	Consider a 1-dimensional Mckean Vlasov equation for $t\in[0,1]$
	\begin{equation*}
		\begin{aligned}
			dX(t) &= 0.1(X(t) +\sin t)dt + 0.1\int_{\mathbb{R}^d} \frac{\sin y}{1+y^2} \mu_t(dy)dt + \frac{1}{\sqrt{10}}X(t) dW_t,\\ 
			X(0) &= X_0 \thicksim N(0,1).
		\end{aligned}
	\end{equation*}
	The corresponding Fokker-Planck equation reads
	\begin{equation*}
		\begin{aligned}
			\frac{\partial p(t,x)}{\partial t}&=0.1\frac{\partial ((x+\sin{t}+\int_{\mathbb{R}}\frac{\sin y}{1+y^2}p(t,y)dy)p(t,x))}{\partial x}+\frac15\frac{\partial^2(x^2p(t,x))}{\partial x^2},\\
			p_0(x) &=\frac{1}{\sqrt{2\pi}}e^{-x^2}.
		\end{aligned}
	\end{equation*}
	We apply a finite difference method to solve this equation and obtain an approximate  density function $p(t,x)$, cf.\ Figure \ref{fig1.5.1}.  
	\begin{figure}[h]
		\centering
		\includegraphics[width=0.5\linewidth]{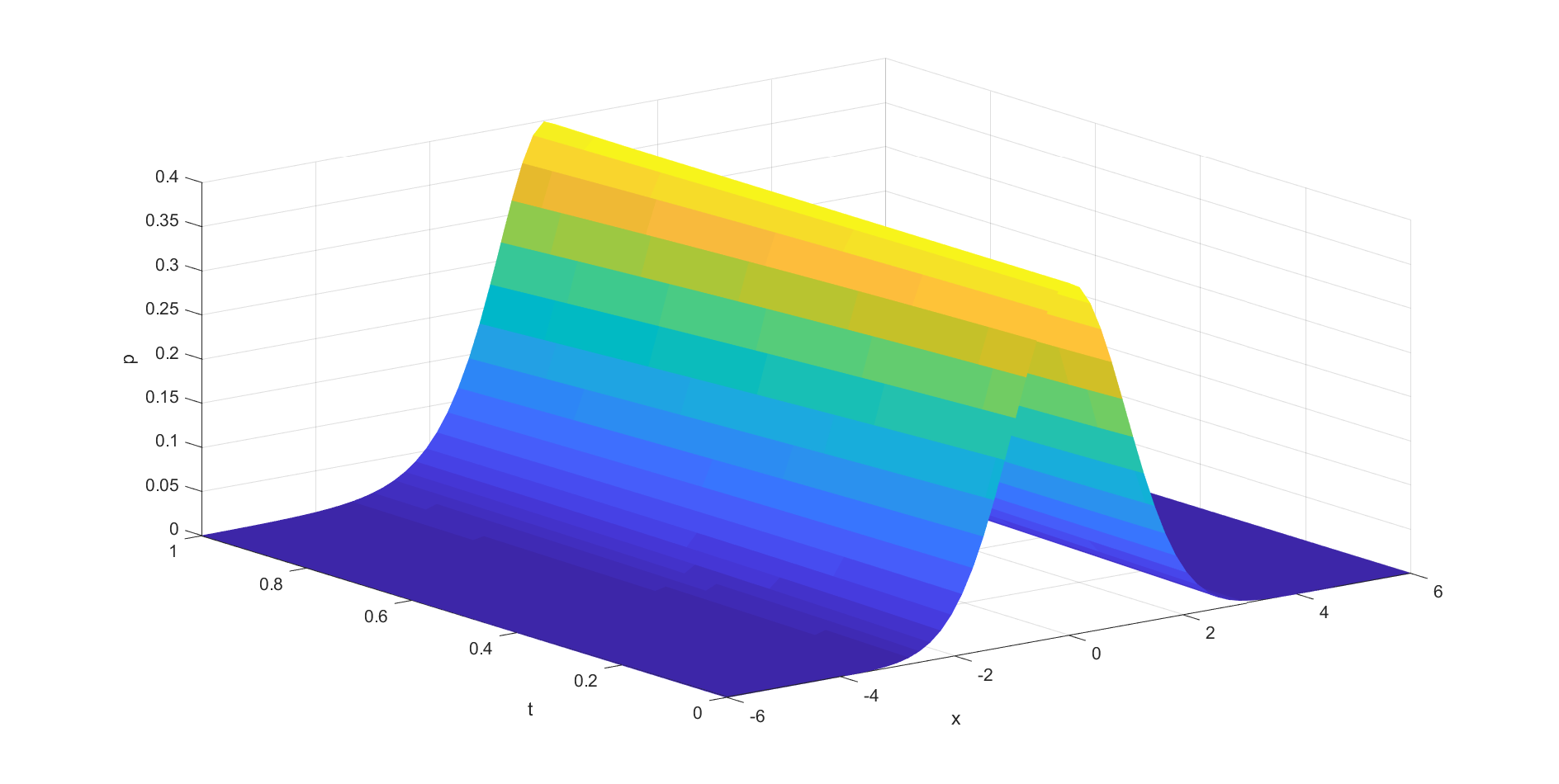}
		\caption{An approximate density function.}
		\label{fig1.5.1}
	\end{figure}
\end{example}    

\begin{example}\label{E1.5.2}
	Consider a 2-dimensional Mckean Vlasov equation 	
	\begin{equation*}
		\begin{aligned}
			dX_1 &\ = \left(\sqrt{(X_1)^2 + (X_2)^2 + 0.4}+\int_{\mathbb{R}^2} \frac{\sin y_1}{1 + y_1^2} \frac{\sin y_2}{1 + y_2^2} \mu_t(dy)\right)dt+\frac{2}{\sqrt{10}}dW^1 + \frac{1}{\sqrt{10}}dW^2, \\
			dX_2 &\ = \left(\sqrt{(X_1)^2 + (X_2)^2 + 0.4}+\int_{\mathbb{R}^2} \frac{\sin y_1}{1 + y_1^2} \frac{\sin y_2}{1 + y_2^2} \mu_t(dy)\right)dt+\frac{1}{\sqrt{10}}dW^1 + \frac{2}{\sqrt{10}}dW^2, \\	
			&\  (X_1(0),X_2(0))  \thicksim N(0, 0.04\cdot I).
		\end{aligned}
	\end{equation*}	
	The corresponding Fokker-Planck equation reads
	\begin{equation*}
		\begin{aligned}
			\frac{\partial p}{\partial t} &\ = 
			-0.1\frac{\partial \big(\sqrt{x_1^2\!+\! x_2^2\!+\! 0.4}\!+\! \int_{\mathbb{R}^2}\frac{\sin y_1}{1+y_1^2}\frac{\sin y_2}{1+y_2^2}p(t,y_1,y_2)dy_1dy_2p\big)
			}{\partial x_1}+\frac14\frac{\partial^2p}{\partial x_1^2}\\
			&\ ~~-0.1\frac{\partial\big(\sqrt{x_1^2\!+\!x_2^2\!+\!0.4}\!+\!\int_{\mathbb{R}^2}\frac{\sin y_1}{1+y_1^2}\frac{\sin y_2}{1+y_2^2}p(t,y_1,y_2)dy_1dy_2 p\big)
			}{\partial x_2}\!+\!\frac14\frac{\partial^2p}{\partial x_2^2}
			\!+\!\frac25\frac{\partial^2p}{\partial x_1\partial x_2}, \\
			&\ p_0(x_1,x_2)=\frac{1}{0.08\pi}e^{-\frac12(\frac{x_1^2}{0.04}+\frac{x_2^2}{0.04})}.
		\end{aligned}
	\end{equation*}
	We approximate the joint density functions for $t\in[0,1]$ and present four of them at different time $t=\frac{1}{4}, \frac{1}{2},\frac{3}{4},1$ in  Figure \ref{figure1.5.2}.
	
	\begin{figure}[h] 
		\centering 
		\subfigure[$t= \frac{1}{4}$.]{
			\includegraphics[width=0.4\textwidth]{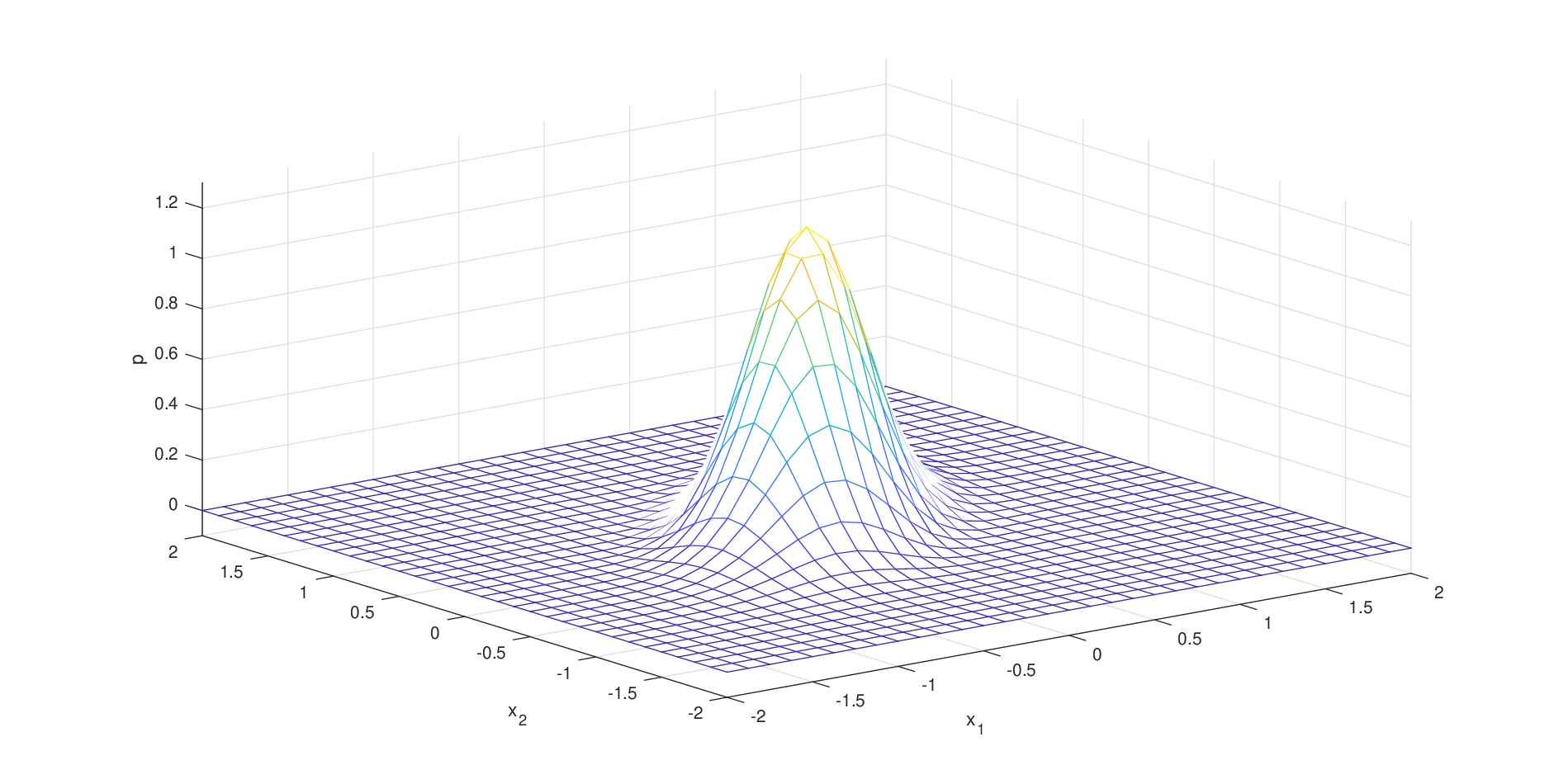}}
		\subfigure[$t=\frac{1}{2}$.]{
			\includegraphics[width=0.4\linewidth]{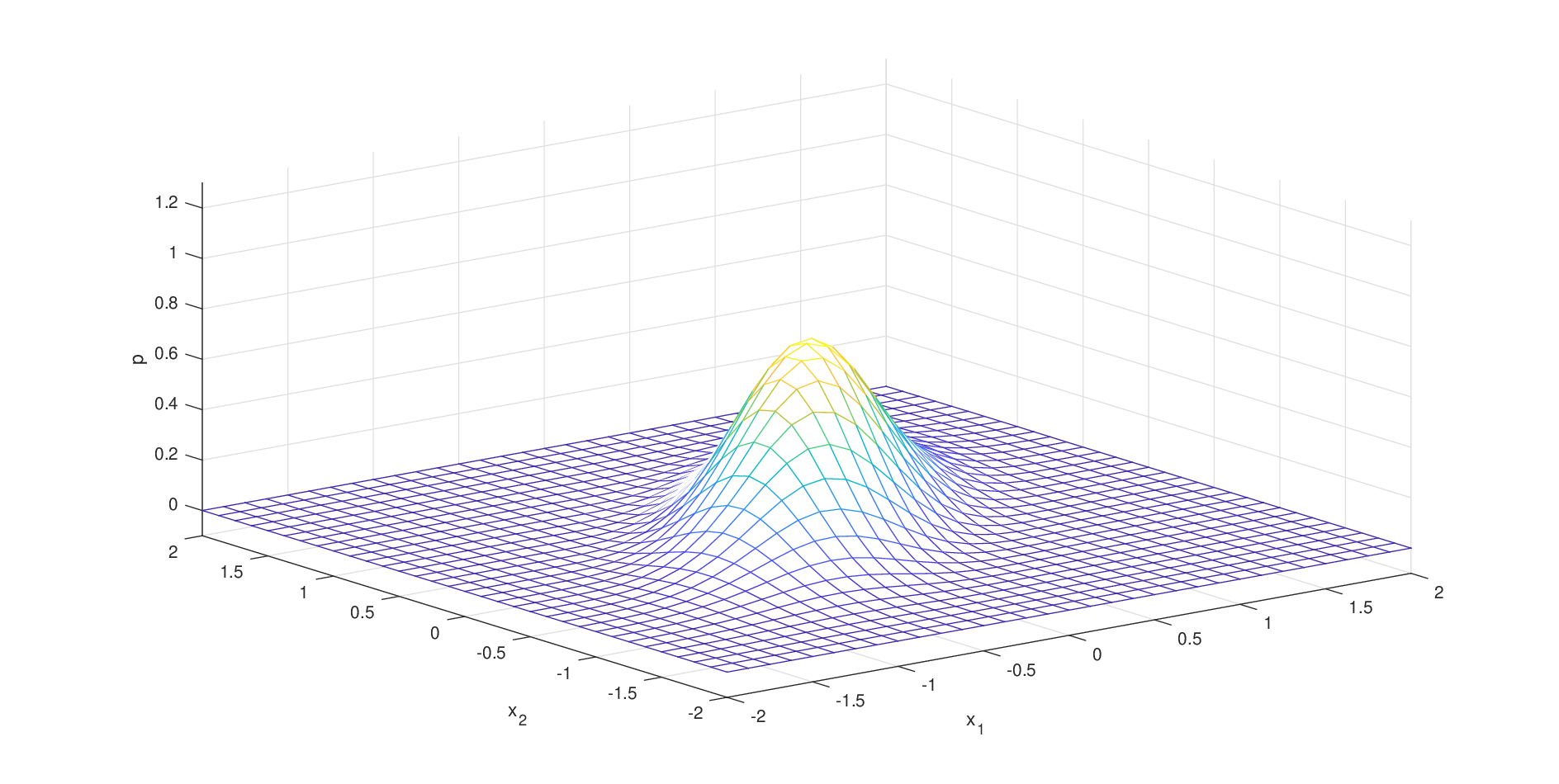}}
		\\
		\subfigure[$t=\frac{3}{4}$.]{
			\includegraphics[width=0.4\linewidth]{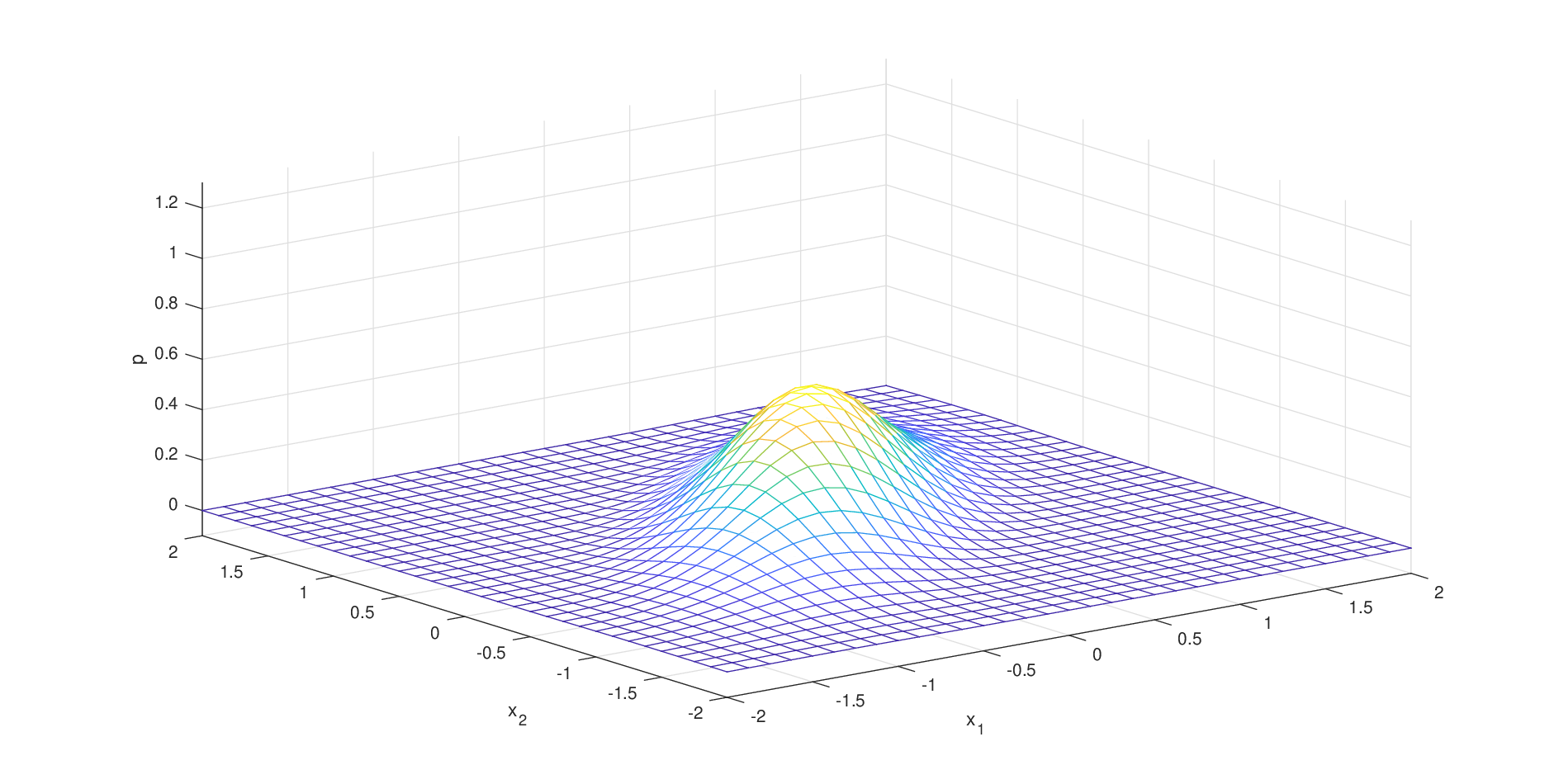}}
		\subfigure[$t=1$.]{
			\includegraphics[width=0.4\linewidth]{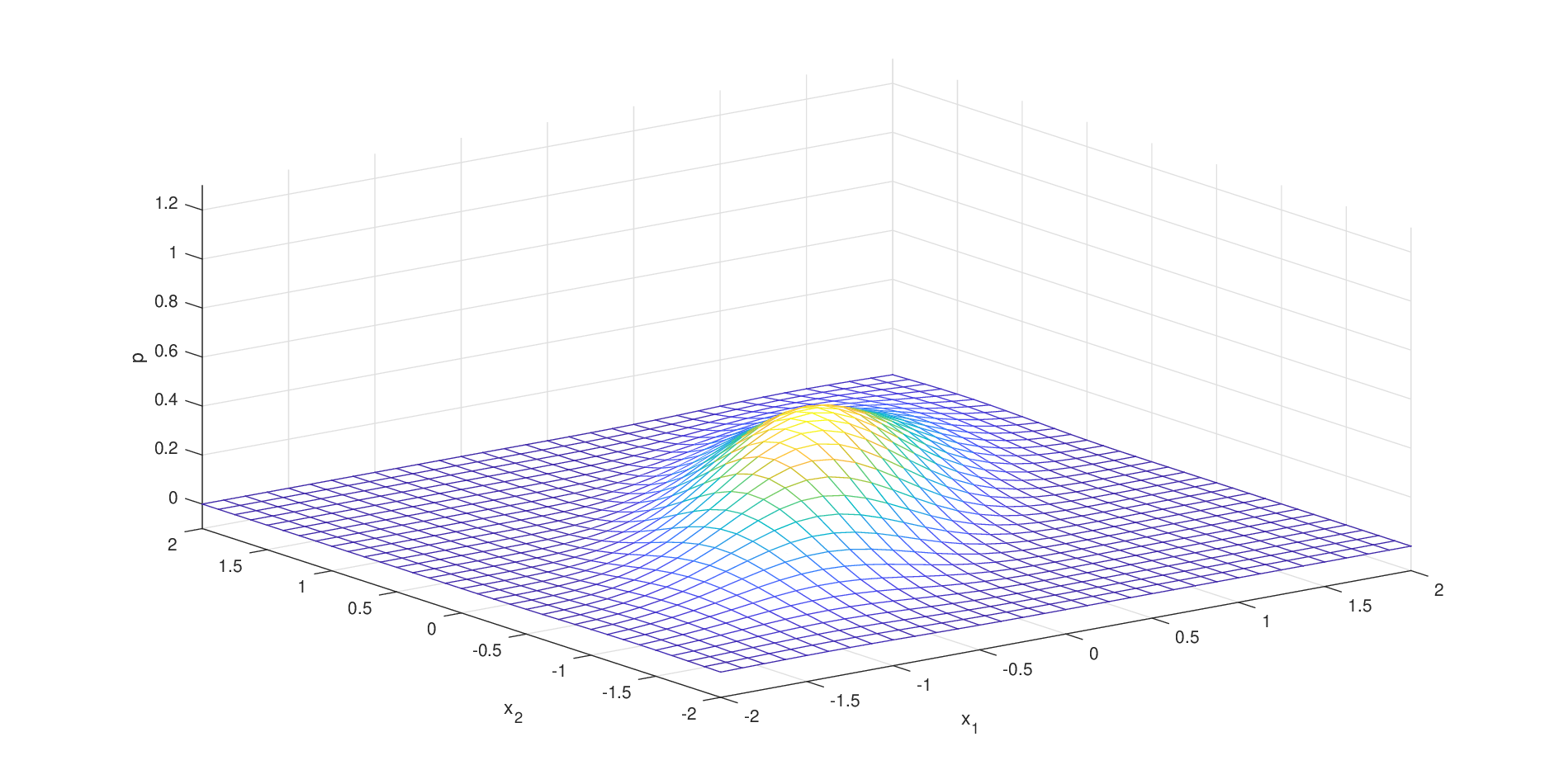}}
		\caption{Development of approximate joint density functions.}
		\label{figure1.5.2}
	\end{figure}
	
\end{example}  

%\section{Conclusions}
%In this paper, we investigate the existence and smoothness of density function to the solution of an MVE. We first obtain the existence of the density function under Lipschitz continuity and uniform elliptic by estimating regularity of the first order Malliavin derivative of solution and the invertibility of its Malliavin covariance matrix.
%Furthermore, we derive the regularity estimate of high order Malliavin derivative by strengthening the smoothness assumption of coefficients, and then obtain the quantitative dependence on the smoothness of coefficients for the regularity of the density function.
%Finally, we present several numerical experiments to illustrate the approximation of the density function independently determined by a Fokker-Planck equation.

\bmhead{Acknowledgments}
This work is supported by Department of Science and Technology of Jilin Province (20240301017GX); National Natural Science Foundation of China (12171199).

\section*{Declarations}
The authors declare that there is no conflict of interest.

\clearpage
\bibliography{ckwx.bib}% common bib file

\end{document}